\documentclass[a4paper,11pt,reqno]{amsart}
\usepackage[utf8]{inputenc}
\usepackage[T1]{fontenc}
\usepackage{lmodern}
\usepackage[english]{babel}
\usepackage{amsmath,a4wide,fullpage}
\usepackage{xfrac}
\usepackage{stmaryrd,mathrsfs,bm,amsthm,mathtools,yfonts,amssymb,color,braket,booktabs,graphicx,graphics,amsfonts}
\usepackage{latexsym,microtype,indentfirst,hyperref}
\usepackage{xcolor}
\usepackage{courier}
\usepackage{hyperref}
\usepackage[colorinlistoftodos]{todonotes}

\begingroup
\newtheorem{theorem}{Theorem}[section]
\newtheorem{lemma}[theorem]{Lemma}
\newtheorem{proposition}[theorem]{Proposition}
\newtheorem{corollary}[theorem]{Corollary}
\endgroup

\theoremstyle{definition}
\newtheorem{definition}[theorem]{Definition}
\newtheorem{remark}[theorem]{Remark}
\newtheorem{ipotesi}[theorem]{Assumption}
\newtheorem{notazioni}[theorem]{Notation}

\newtheorem{assumptions}[theorem]{Assumptions}

\numberwithin{equation}{section}
\numberwithin{subsection}{section}

\newcommand{\N}{\mathbb{N}}
\newcommand{\Z}{\mathbb{Z}}

\newcommand{\R}{\mathbb{R}}

\newcommand{\mres}{\mathbin{\vrule height 1.6ex depth 0pt width
0.13ex\vrule height 0.13ex depth 0pt width 1.3ex}}

\newcommand{\U}{\mathcal{U}}

\newcommand{\A}{\mathcal{A}}
\newcommand{\B}{\mathcal{B}}
\newcommand{\G}{\mathcal{G}}

\newcommand{\D}{\mathcal{D}}
\newcommand{\M}{\mathbf{M}}
\newcommand{\Rc}{\mathscr{R}}
\newcommand{\I}{\mathscr{I}}
\newcommand{\Fl}{\mathbf{F}}
\newcommand{\F}{\mathscr{F}}

\newcommand{\K}{\mathcal{K}}

\newcommand{\Lip}{\mathrm{Lip}}
\newcommand{\spt}{\mathrm{spt}}
\newcommand{\dist}{\mathrm{dist}}

\newcommand{\Ha}{\mathcal{H}}

\newcommand{\bphi}{\boldsymbol{\varphi}}
\newcommand{\bfeta}{\boldsymbol{\eta}}

\newcommand{\p}{\mathbf{p}}

\newcommand{\weakto}{\rightharpoonup}

\title[Multiple valued sections of vector bundles]{Multiple valued sections of vector bundles: the reparametrization theorem for $Q$-valued functions revisited}

\author{Salvatore Stuvard}

\newcommand{\Addresses}{{
  \bigskip
  \footnotesize

  S. S., \textsc{Universit\"at Z\"urich, Winterthurerstrasse 190, CH-8057 Z\"urich, Switzerland}
  \par\nopagebreak
  
\bigskip 
 
  \textit{E-mail address}, S. S.: \href{stuvard@math.utexas.edu}{stuvard@math.utexas.edu}
   
}}

\begin{document}

\begin{abstract}
We analyze a notion of multiple valued sections of a vector bundle over an abstract smooth Riemannian manifold, which was suggested by W. Allard in the unpublished note ``\textit{Some useful techniques for dealing with multiple valued functions}'' and generalizes Almgren's $Q$-valued functions. We study some relevant properties of such $Q$-multisections and apply the theory to provide an elementary and purely geometric proof of a delicate reparametrization theorem for multi-valued graphs which plays an important role in the regularity theory for higher codimension area minimizing currents \`a la Almgren-De Lellis-Spadaro.

\vspace{4pt}

\vspace{4pt}

\end{abstract}

\maketitle

\section{Introduction}

Introduced by F. Almgren in his groundbreaking monograph \cite{Almgren00}, multiple valued functions are an indispensable tool to address the regularity problem for area minimizing currents in codimension higher than one. In recent years, C. De Lellis and E. Spadaro have brought to a successful conclusion the challenging project to revisit Almgren's regularity theory, taking advantage of the tools from metric analysis and metric geometry developed in the last couple of decades in order to substantially reduce the complexity of the subject and give a new insight of the whole theory itself, cf. \cite{DLS11a, DLS13a, DLS14, DLS13b, DLS13c} and also \cite{DeLellis2015}. Since then, the techniques pioneered by De Lellis and Spadaro have been fruitfully applied to a wide variety of problems, including the regularity of $2$-dimensional integral currents that are (in a suitable sense) almost minimizing the area functional \cite{DLSS1,DLSS2,DLSS3}, the boundary regularity of area minimizing integral currents \cite{DLDPHM}, and the regularity of integer rectifiable currents that minimize the mass in their homology class modulo $p$ \cite{DLHMS_linear, DLHMS}. 

Understanding the connection between multiple valued functions and integer rectifiable currents is crucial to carry on the Almgren-De Lellis-Spadaro program. A basic observation is that one can naturally associate an integer rectifiable current to the graph of a Lipschitz multiple valued function. This can be done by defining a suitable notion of \emph{push-forward} of a Lipschitz manifold through a multiple valued function, see \cite{DLS13a} and \S\, \ref{sec:pf} below. On the other hand, a highly non-trivial procedure allows one to approximate the rescalings of an area minimizing current at an interior branch point of density $Q$ with the graphs of a sequence of $Q$-valued functions which converge, in the limit, to a $Q$-valued function which is $Q$-harmonic, in the sense that it minimizes a conveniently defined Dirichlet energy. This fact is the key to reduce the regularity problem for area minimizing currents to the regularity problem for ${\rm Dir}$-minimizing $Q$-valued functions. 

When performing the above approximation procedure, it is crucial that the limiting ${\rm Dir}$-minimizer ``inherits'' the singularities of the current. In order to guarantee that this happens, it is necessary to suitably construct a regular manifold (the \emph{center manifold}) which is an approximate ``average'' of the sheets of the current itself, and to approximate with high degree of accuracy the current with $Q$-valued functions defined on the center manifold and taking values in its normal bundle. This goal is achieved in \cite{DLS13b}. The key step is to derive a result concerning the possibility to \emph{reparametrize} the graph of a Lipschitz multiple valued function. Specifically, the problem of interest here is the following: let $f \colon \Omega \subset \R^{m} \to \A_{Q}(\R^{n})$ be a Lipschitz $Q$-valued function, and let $\Sigma$ be a regular manifold which is the graph of a sufficiently smooth function $\bphi \colon \Omega' \subset \Omega \to \R^{n}$. If the Lipschitz constant of $f$ is small and $\Sigma$ is sufficiently flat, then is it possible to represent the graph of $f$ also as the image of a Lipschitz multiple valued function $F$ defined on $\Sigma$ and taking values in its normal bundle? Furthermore, which control do we have on the Lipschitz constant of $F$ in terms of the Lipschitz constant of $f$ and of the geometry of $\Sigma$? 

Such a problem has been tackled in \cite{DLS13a}, where the authors apply the theory of currents in metric spaces à la Ambrosio-Kirchheim (see \cite{AK00}) to successfully prove the reparametrization theorem needed in \cite{DLS13b}.

The ultimate goal of this note is to provide a completely elementary and purely geometric proof of such a reparametrization theorem for Lipschitz multiple valued functions, without making use of the Ambrosio-Kirchheim theory (see Theorem \ref{reparametrization:thm}). This is achieved by developing a theory of multiple valued \emph{sections} (\emph{$Q$-multisections}) of an abstract vector bundle $\Pi \colon E \to \Sigma$ over an abstract smooth Riemannian manifold, stemming from some unpublished ideas of W. Allard \cite{Allard} and generalizing the notion of $Q$-valued function. Two properties of \emph{coherence} and \emph{vertical boundedness} for a $Q$-multisection are particularly relevant, as they ``mimic'' the classical Lipschitz continuity in the vector bundle-valued case (see Propositions \ref{continuity:prop} and \ref{Lip:prop}).

The theory of $Q$-multisections seems to be of independent interest, yet to be fully developed and capable of further applications. As in the single-valued case, indeed, it is often of interest to minimize a given functional of the Calculus of Variations among multiple valued functions which are constrained to take values in some vector bundle over a given manifold (see, for instance, our paper \cite{SS17b}, where we develop a multivalued theory for the stability operator). We strongly believe that the theory of $Q$-multisections may provide useful tools to deal with similar situations.

\subsection{Content of the paper.}

This note is organized in three sections and two appendices. In Section \ref{sec:prel} we provide a quick tutorial on multiple valued functions and integer rectifiable currents, recall the relevant results that are used in the rest of the paper and fix terminology and notation; Section \ref{sec:sections} contains the definition of $Q$-multisections, and a detailed analysis of the aforementioned properties of coherence and vertical boundedness; our new approach to $Q$-valued reparametrizations is finally presented in Section \ref{sec:reparam}. The two appendices contain some further results on the theory of push-forwards through multiple valued functions: in particular, in Appendix \ref{sec:commutation} we present a slightly simplified proof (with respect to \cite[Theorem 2.1]{DLS13a}) of the fact that the multi-valued push-forward operator on Lipschitz submanifolds commutes with the boundary operator in the sense of currents; in Appendix \ref{ssec:pf_fc} we extend the notion of multi-valued push-forward to the class of integral flat chains.

\medskip

\noindent\textbf{Acknowledgments.} The author is very grateful to Camillo De Lellis for his invaluable suggestions and constant support, to William Allard for sharing with him some of his beautiful ideas, and to Andrea Marchese for carefully reading a preliminary version of this manuscript and for his very helpful comments. The author is also thankful to the anonymous referees for their valuable suggestions.

The research of S.S. has been supported by the ERC grant agreement RAM (Regularity for Area Minimizing currents), ERC 306247.

\section{Preliminaries} \label{sec:prel}

We recall here the basic facts concerning multiple valued functions and integer rectifiable currents, mainly in order to fix the notation that will be used throughout the paper. Our main reference for multiple valued functions will be \cite{DLS11a}, where De Lellis and Spadaro revisit and simplify Almgren's original theory in \cite{Almgren00}. For a thorough discussion about currents, instead, the reader can refer to standard books in Geometric Measure Theory such as \cite{Sim83} and \cite{KP08}, to the monograph \cite{GMS98} or to the treatise \cite{Federer69}.

\subsection{The metric space of $Q$-points.} Let $Q$ be a fixed positive integer. We will denote by $\A_{Q}(\R^n)$ the set of $Q$-points in $\R^n$, defined by
\[
\A_{Q}(\R^n) = \left\lbrace T = \sum_{l=1}^{Q} \llbracket v_{l} \rrbracket \, \colon \, \mbox{each } v_{l} \in \R^n \right\rbrace,
\]
where $\llbracket v \rrbracket$ denotes the Dirac delta $\delta_{v}$ centered at $v \in \R^n$.

%


The set $\A_{Q}(\R^n)$ will be equipped with the structure of metric space: the distance between two $Q$-points is defined as the Wasserstein distance of exponent two between the associated measures (cf. for instance \cite[Section 7.1]{Vil03}). Specifically, if $T_{1} = \sum_{l=1}^{Q} \llbracket v_{l} \rrbracket$ and $T_{2} = \sum_{l=1}^{Q} \llbracket w_{l} \rrbracket$, then the distance between $T_{1}$ and $T_{2}$ is the quantity
\[
\G(T_{1}, T_{2}) := \left( \min_{\sigma \in \mathcal{P}_{Q}} \sum_{l=1}^{Q} |v_{l} - w_{\sigma(l)}|^{2} \right)^{\frac{1}{2}}\,,
\]
where $\mathcal{P}_{Q}$ is the group of the permutations of $\{1,\dots,Q\}$. One can easily see that $\left( \A_{Q}(\R^n), \G \right)$ is a complete, separable metric space.

If $T \in \A_{Q}(\R^{n})$ can be written as $T = m \llbracket v \rrbracket + \sum_{i=1}^{Q-m} \llbracket v_{i} \rrbracket$ with each $v_{i} \neq v$, then we say that $v$ has \emph{multiplicity} $m$ in $T$. Sometimes, when $v$ has multiplicity $m$ in $T$ we will write $m = \Theta_{T}(v)$, using a notation which is coherent with regarding $T$ as a $0$-dimensional integer rectifiable current in $\R^{n}$ (see \cite[Section 27]{Sim83} and Remark \ref{rmk:deltas} below).

Also, to any point $T = \sum_{l} \llbracket v_{l} \rrbracket \in \A_{Q}(\R^{n})$ one can naturally associate two objects, of which we will make use in the sequel: the \emph{diameter} of $T$ is the scalar
\[
{\rm diam}(T) := \max_{i,j \in \{1,\dots,Q\}} |v_{i} - v_{j}|,
\]
whereas the \emph{center of mass} of $T$ is the vector
\[
\bfeta(T) := \frac{1}{Q} \sum_{l=1}^{Q} v_{l}.
\]

\subsection{$Q$-valued functions.} Let $\Sigma = \Sigma^{m}$ be an $m$-dimensional $C^{1}$ submanifold of $\R^{d}$. In what follows, integrals on $\Sigma$ will always be computed with respect to the $m$-dimensional Hausdorff measure $\Ha^{m}$ defined in the ambient space. A \emph{$Q$-valued function} on $\Sigma$ is any map $u \colon \Sigma \to \A_{Q}(\R^n)$. 
If $B \subset \Sigma$ is a measurable subset, every measurable function $u \colon B \to \A_{Q}(\R^n)$ can be thought as coming together with a measurable selection, i.e. a $Q$-tuple of measurable functions $u_1,\dots,u_Q \colon B \to \R^n$ such that
\begin{equation}
u(p) = \sum_{l=1}^{Q} \llbracket u_{l}(p) \rrbracket \qquad \mbox{ for $\Ha^m$-a.e. } p \in B\,,
\end{equation}
see \cite[Proposition 0.4]{DLS11a}.
%
%

The metric space structures of both $\Sigma$ and $\A_{Q}(\R^n)$ allow to straightforwardly define H\"older and Lipschitz continuous $Q$-valued functions on $\Sigma$. Moreover, a notion of \emph{differentiability} can be introduced for $u \colon \Sigma \to \A_{Q}(\R^n)$ as follows.
\begin{definition}[Differentiable $Q$-valued functions] \label{differentiable}
A map ${u \colon \Sigma \to \A_{Q}(\R^n)}$ is \emph{differentiable} at $p \in \Sigma$ if there exist $Q$ linear maps $\lambda_{l} \colon T_{p}\Sigma \to \R^n$ satisfying:
\begin{itemize}
\item[$(i)$] $\G\left( u(\exp_{p}(\tau)), T_{p}u(\tau) \right) = o(|\tau|)$ as $|\tau| \to 0$ for any $\tau \in T_{p}\Sigma$, where $\exp$ is the exponential map on $\Sigma$ and 
\begin{equation}
T_{p}u(\tau) := \sum_{l=1}^{Q} \llbracket u_{l}(p) + \lambda_{l} \cdot \tau \rrbracket\,;
\end{equation}
\item[$(ii)$] $\lambda_{l} = \lambda_{l'} \mbox{ if } u_{l}(p) = u_{l'}(p)$.
\end{itemize}
\end{definition}
We will use the notation $Du_{l}(p)$ for $\lambda_{l}$, and formally set $Du(p) = \sum_{l} \llbracket Du_{l}(p) \rrbracket$: observe that one can regard $Du(p)$ as an element of $\A_{Q}(\R^{n \times m})$ as soon as a basis of $T_{p}\Sigma$ has been fixed. For any $\tau \in T_{p}\Sigma$, we define the directional derivative of $u$ along $\tau$ to be $D_{\tau}u(p) := \sum_{l} \llbracket Du_{l}(p) \cdot \tau \rrbracket$, and establish the notation $D_{\tau}u = \sum_{l} \llbracket D_{\tau}u_l \rrbracket$.

A version of Rademacher's theorem can be proved in this setting, and thus Lipschitz $Q$-valued functions turn out to be differentiable in the sense of the above definition at $\Ha^{m}$-a.e. $p$ (cf. \cite[Theorem 1.13]{DLS11a}). Furthermore, the result concerning the existence of measurable selections can be improved, as Lipschitz $Q$-valued functions enjoy the following \emph{Lipschitz selection property}.
\begin{proposition}[Lipschitz selection, cf. {\cite[Lemma 1.1]{DLS13a}}] \label{Lip_sel}
Let $B \subset \Sigma$ be measurable, and assume $u \colon B \to \A_{Q}(\R^n)$ is Lipschitz. Then, there are a countable partition of $B$ in measurable subsets $B_{i}$ ($i \in \N$) and Lipschitz functions $u_{i}^{l} \colon B_{i} \to \R^n$ ($l \in \{1,\dots,Q\}$) such that
\begin{itemize}
\item[$(a)$] $u|_{B_i} = \sum_{l=1}^{Q} \llbracket u_{i}^{l} \rrbracket$ for every $i \in \N$, and $\Lip(u_{i}^{l}) \leq \Lip(u)$ for every $i$,$l$;
\item[$(b)$] for every $i \in \N$ and $l, l' \in \{1,\dots,Q\}$, either $u_{i}^{l} \equiv u_{i}^{l'}$ or $u_{i}^{l}(p) \neq u_{i}^{l'}(p) \, \forall \, p \in B_{i}$;
\item[$(c)$] for every $i$ one has $Du(p) = \sum_{l=1}^{Q} \llbracket Du_{i}^{l}(p) \rrbracket$ for a.e. $p \in B_{i}$.
\end{itemize}
\end{proposition}


\subsection{General currents: an overview.} 



\medskip

Given an open set $\Omega \subset \R^d$, an $m$-dimensional \emph{current} in $\Omega$ is a linear and continuous functional
\[
T \colon \D^{m}(\Omega) \to \R,
\]
where $\D^{m}(\Omega)$ denotes the space of smooth and compactly supported differential $m$-forms in $\Omega$, equipped with the standard locally convex topology of $C^{\infty}_{c}(\Omega, \Lambda^{m}(\R^{d}))$, $\Lambda^{m}(\R^{d})$ being the vector space of $m$-covectors in $\R^{d}$ (cf. \cite[Sections 25 and 26]{Sim83}).

The space of $m$-currents in $\Omega$ is therefore the topological dual space of $\D^{m}(\Omega)$, and will be denoted by $\D_{m}(\Omega)$. Observe that if $\Sigma \subset \Omega$ is an oriented $m$-dimensional submanifold, then there is a corresponding $m$-current $\llbracket \Sigma \rrbracket \in \D_{m}(\Omega)$ defined by integration of $m$-forms on $\Sigma$ in the usual sense of differential geometry:
\[
\llbracket \Sigma \rrbracket(\omega) := \int_{\Sigma} \omega \hspace{0.5cm} \forall \, \omega \in \D^{m}(\Omega).
\] 

\begin{remark} \label{rmk:deltas}
In particular, if $p \in \Omega$ then the action of the $0$-dimensional current associated to $p$ is given by
\[
\llbracket p \rrbracket (f) = f(p) \quad \forall \, f \in \D^{0}(\Omega) = C^{\infty}_{c}(\Omega),
\] 
and thus $\llbracket p \rrbracket$ is the Dirac delta $\delta_{p}$ centered at $p$ and acting on smooth and compactly supported functions. Therefore, the notation here adopted for the current associated to a submanifold is coherent with that already used before to denote the $Q$-points in Euclidean space.
\end{remark}

If $T$ is an $m$-currents, the symbol $\partial T$ will denote its \emph{boundary}, namely the $(m-1)$-current whose action on any form $\omega \in \D^{m-1}(\Omega)$ is given by
\[
\partial T(\omega) := T({\rm d}\omega),
\]
where ${\rm d}\omega$ is the exterior differential of $\omega$.


The \emph{mass} of $T \in \D_{m}(\Omega)$, denoted $\M(T)$, is the (possibly infinite) supremum of the values $T(\omega)$ among all forms $\omega \in \D^{m}(\Omega)$ with $\|\omega(p)\|_{c} \leq 1$ everywhere\, \footnote{Here, the symbol $\| \omega \|_{c}$ denotes the \emph{comass} of the covector $\omega$ (cf. \cite[Section 25]{Sim83}).}. If the class of competitors in the supremum is restricted only to those forms $\omega$ with $\spt(\omega) \subset W$ for some $W \Subset \Omega$ then we obtain the \emph{localized mass} of $T$ in $W$, denoted $\M_W(T)$.

%
%
%

The \emph{support} $\spt(T)$ of the current $T$ is the intersection of all closed subsets $C$ such that $T(\omega) = 0$ whenever $\spt(\omega) \subset \R^{d} \setminus C$.

Observe that if $T = \llbracket \Sigma \rrbracket$ is the current associated with an oriented $m$-dimensional submanifold $\Sigma$ then $\partial \llbracket \Sigma \rrbracket = \llbracket \partial \Sigma \rrbracket$ (here, $\partial \Sigma$ is oriented by the Stokes' orientation induced by $\Sigma$), $\M(\llbracket \Sigma \rrbracket) = \Ha^m(\Sigma)$ and $\spt(\llbracket \Sigma \rrbracket) = \overline{\Sigma}$. 

A suitable notion of convergence of currents can be defined by endowing $\D_{m}(\Omega)$ with the weak-$^{*}$ topology induced by $\D^{m}(\Omega)$. Hence, we will say that a sequence $\{T_{h}\}_{h=1}^{\infty} \subset \D_{m}(\Omega)$ \emph{converges} to $T \in \D_{m}(\Omega)$ \emph{in the sense of currents}, and we will write $T_{h} \weakto T$, if $T_{h}(\omega) \to T(\omega)$ for every $\omega \in \D^{m}(\Omega)$. It is clear that if $T_{h} \weakto T$ then also $\partial T_{h} \weakto \partial T$. Moreover, the mass is lower semi-continuous with respect to convergence in the sense of currents.

\subsection{Classes of currents.} We will denote by $\Rc_m(\Omega)$ the set of \emph{integer rectifiable $m$-currents} in $\Omega$, that is the set of currents $T$ whose action on forms is given by
\begin{equation} \label{corrente rettificabile}
T(\omega) := \int_{B} \langle \omega(p), \vec{\tau}(p) \rangle \, \theta(p) \, d\Ha^{m}(p) \hspace{0.5cm} \forall \, \omega \in \D^{m}(\Omega)\,,
\end{equation}
where:
\begin{itemize}
\item[•] $B \subset \Omega$ is a (countably) $m$-rectifiable set, i.e. $\Ha^m(B \cap K) < \infty$ for any compact $K \subset \Omega$, and moreover $B$ can be covered up to a $\Ha^{m}$-null set by countably many $m$-dimensional embedded submanifolds of $\R^d$ of class $C^1$;
\item[•] the function $\vec{\tau} \colon B \to \Lambda_{m}(\R^d)$, where $\Lambda_{m}(\R^{d})$ denotes the vector space of $m$-vectors in $\R^{d}$, is a measurable \emph{orientation} of $B$, namely $\vec{\tau}(p) = \tau_{1}(p) \wedge \dots \wedge \tau_{m}(p)$, where $\left( \tau_{1}(p), \dots, \tau_{m}(p) \right)$ is an orthonormal basis of the approximate tangent space ${\rm Tan}(B,p)$ at $\Ha^m$-a.e. $p \in B$;
\item[•] $\theta \in L^{1}_{loc}(B,\Z;\Ha^m \mres B)$.
\end{itemize}

We write $T = \llbracket B, \vec{\tau}, \theta \rrbracket$ if $T$ is an integer rectifiable current as in \eqref{corrente rettificabile} .

Integer rectifiable currents with integer rectifiable boundary are called \emph{integral currents}. We write $T \in \I_{m}(\Omega)$ if $T$ is an integral $m$-current in $\Omega$. 


If $K \subset \Omega$ is a compact set, we will denote by $\Rc_{m,K}(\Omega)$ (resp. $\I_{m,K}(\Omega)$) the set of integer rectifiable (resp. integral) $m$-currents $T$ with $\spt(T) \subset K$. We also set
\[
\F_{m,K}(\Omega) := \left\lbrace T = R + \partial S \, \colon \, R \in \Rc_{m,K}(\Omega) \, \mbox{ and } \, S \in \Rc_{m+1,K}(\Omega) \right\rbrace,
\]
and we let $\F_{m}(\Omega)$ be the union of the sets $\F_{m,K}(\Omega)$ over all compact $K \subset \Omega$. Currents $T \in \F_{m}(\Omega)$ are called $m$-dimensional (integral) \emph{flat chains} in $\Omega$. On each set $\F_{m,K}(\Omega)$ one can define a metric as follows:  for $T \in \F_{m,K}(\Omega)$, set
\[
\Fl_{K}(T) := \inf\left\lbrace \M(R) + \M(S) \, \colon \, R \in \Rc_{m,K}(\Omega), S \in \Rc_{m+1,K}(\Omega) \, \mbox{such that } \, T = R + \partial S \right\rbrace,
\] 
and then let the distance between $T_{1}$ and $T_{2}$ (usually called \emph{flat distance}) be given by
\[
d_{\Fl_{K}}(T_{1}, T_{2}) := \Fl_{K}(T_{1} - T_{2}).
\]
It turns out that the resulting metric space $\left( \F_{m,K}(\Omega), d_{\Fl_K} \right)$ is complete. Moreover, the mass functional is lower semi-continuous with respect to the \emph{flat convergence}. It is immediate to show that if a sequence $\{T_{h}\}$ of flat chains converges to $T$ with respect to the flat distance then it also weakly converges to $T$. The two notions of convergence are in fact equivalent if $\{ T_{h} \}$ is a sequence of integral currents with equi-bounded masses and masses of the boundaries (cf. \cite[Theorem 31.2]{Sim83}).

\subsection{Push-forwards through multiple valued functions} \label{sec:pf}

Let $T \in \D_{m}(\Omega)$ and suppose $f \colon \Omega \to \R^{n}$ is a $C^{\infty}$ map. If $f$ is \emph{proper} (i.e. $f^{-1}(K)$ is compact for any compact $K \subset \R^{n}$), then the \emph{push-forward} of $T$ through $f$ is the current $f_{\sharp}T \in \D_{m}(\R^{n})$ defined by
\[
f_{\sharp}T(\omega) := T(f^{\sharp}\omega) \hspace{0.5cm} \forall \, \omega \in \D^{m}(\R^{n}),
\] 
where $f^{\sharp}\omega$ denotes the pull-back of the form $\omega$ through $f$. The push-forward operator $f_{\sharp}$ is linear, and moreover an elementary computation shows that it commutes with the boundary operator:
\[
\partial(f_{\sharp} T) = f_{\sharp}(\partial T)\,.
\]

\medskip

If $T = \llbracket B, \vec{\tau}, \theta \rrbracket$ is rectifiable, then it is straightforward to verify that the push-forward $f_{\sharp}T$ is given explicitly by
\[
f_{\sharp}T(\omega) = \int_{B} \langle \omega(f(p)), Df(p)_{\sharp}\vec{\tau}(p) \rangle \, \theta(p) \, d\Ha^{m}(p) \hspace{0.5cm} \forall \, \omega \in \D^{m}(\R^{n}), 
\]
where 
\[
Df(p)_{\sharp}\vec{\tau}(p) := (Df(p) \cdot \tau_{1}(p)) \wedge \dots \wedge (Df(p) \cdot \tau_{m}(p)).
\]
The hypotheses on $f$ can in fact be relaxed, as the above formula makes sense whenever $f \colon B \to \R^{n}$ is Lipschitz and proper. In this case, $Df(p)$ has to be regarded as the tangent map of $f$ at $p$, which exists at $\Ha^{m}$-a.e. $p \in B$ since $B$ is rectifiable and $f$ is Lipschitz. Also, it is not difficult to show that in this case $f_{\sharp}T$ is also $m$-rectifiable, and in fact $f_{\sharp}T = \llbracket f(B), \vec{\eta}, \Theta \rrbracket$, where $\vec{\eta}(y) = \eta_1(y) \wedge \dots \wedge \eta_m(y)$ is a simple unit $m$-vector field orienting ${\rm Tan}(f(B),y)$ at $\Ha^m$-a.e. $y \in f(B)$ and
\[
\Theta(y) := \sum_{p \in B_{+} \colon f(p) = y} \theta(p) \left\langle \vec{\eta}(y), \frac{Df(p)_{\sharp}\vec{\tau}(p)}{|Df(p)_{\sharp}\vec{\tau}(p)|} \right\rangle\,, \qquad B_+ := \left\lbrace p \in B \, \colon \, {\bf J}f(p) > 0 \right\rbrace\,,
\]
having denoted ${\bf J}f(p) = |Df(p)_{\sharp}\vec{\tau}(p)|$ the tangential Jacobian of $f$.

\medskip

In \cite{DLS13a}, the authors take advantage of the Lipschitz selection property stated in Proposition \ref{Lip_sel} above in order to tackle the problem of extending the above results to the context of multiple valued functions.

More precisely, let us assume that $\Sigma$ is an $m$-dimensional $C^{1}$ submanifold of $\R^{d}$, and $B \subset \Sigma$ is $\Ha^{m}$-measurable. Recall that a measurable function $u \colon B \subset \Sigma \to \A_{Q}(\R^{n})$ is \emph{proper} if there exists a measurable selection $u = \sum_{l=1}^{Q} \llbracket u_{l} \rrbracket$ such that the set $\bigcup_{l=1}^{Q} \overline{u_{l}^{-1}(K)}$ is compact for any compact $K \subset \R^{n}$.


\begin{definition}[$Q$-valued push-forward, cf. {\cite[Definition 1.3]{DLS13a}}] \label{Q_pf:def}
Let $B \subset \Sigma$ be as above, and let $u \colon B \to \A_{Q}(\R^n)$ be Lipschitz and proper. Then, the \emph{push-forward} of $B$ through $u$ is the current $\mathbf{T}_{u} := \sum_{i \in \N} \sum_{l=1}^{Q} (u_{i}^{l})_{\sharp}\llbracket B_i \rrbracket$, where $B_i$ and $u_{i}^{l}$ are as in Proposition \ref{Lip_sel}: that is,
\begin{equation} \label{Q_pf:eq}
\mathbf{T}_{u}(\omega) := \sum_{i \in \mathbb{N}} \sum_{l=1}^{Q} \int_{B_i} \left\langle \omega(u_{i}^{l}(p)), Du_{i}^{l}(p)_{\sharp}\vec{\tau}(p) \right\rangle \, d\Ha^m(p) \hspace{0.5cm} \forall \, \omega \in \D^{m}(\R^n).
\end{equation}
\end{definition}

Using the classical results concerning the push-forward of integer rectifiable currents through (single valued) proper Lipschitz functions recalled above and the properties of Lipschitz selections, it is not difficult to conclude that Definition \ref{Q_pf:def} is well-posed.
\begin{proposition}[Representation of the push-forward, cf. {\cite[Proposition 1.4]{DLS13a}}] \label{Q_pf:repr}
The definition of the action of $\mathbf{T}_{u}$ in \eqref{Q_pf:eq} does not depend on the chosen partition $B_i$, nor on the chosen decomposition $\{u_{i}^{l}\}$. If $u = \sum_{l} \llbracket u_{l} \rrbracket$, we are allowed to write
\begin{equation}
\mathbf{T}_{u}(\omega) = \int_{B} \sum_{l=1}^{Q} \left\langle \omega(u_{l}(p)), Du_{l}(p)_{\sharp}\vec{\tau}(p) \right\rangle \, d\Ha^{m}(p) \hspace{0.5cm} \forall \, \omega \in \mathcal{D}^{m}(\R^n).
\end{equation}
Thus, $\mathbf{T}_{u}$ is a (well-defined) integer rectifiable $m$-current in $\R^n$ given by $\mathbf{T}_{u} = \llbracket \mathrm{Im}(u),  \vec{\eta}, \Theta \rrbracket$, where:
\begin{itemize}
\item[$(R1)$] $\mathrm{Im}(u) = \bigcup_{p \in B} \spt(u(p)) = \bigcup_{i \in \mathbb{N}} \bigcup_{l=1}^{Q} u_{i}^{l}(B_i)$ is an $m$-rectifiable set in $\R^n$;
\item[$(R2)$] $\vec{\eta}$ is a Borel unit $m$-vector field orienting $\mathrm{Im}(u)$; moreover, for $\Ha^m$-a.e. $y \in \mathrm{Im}(u)$, we have $Du_{i}^{l}(p)_{\sharp}\vec{\tau}(p) \neq 0$ for every $i,l,p$ such that $u_{i}^{l}(p) = y$ and
\begin{equation}
\vec{\eta}(y) = \pm \frac{Du_{i}^{l}(p)_{\sharp}\vec{\tau}(p)}{|Du_{i}^{l}(p)_{\sharp}\vec{\tau}(p)|};
\end{equation}
\item[$(R3)$] for $\Ha^m$-a.e. $y \in \mathrm{Im}(u)$, the (Borel) multiplicity function $\Theta$ equals
\begin{equation}
\Theta(y) = \sum_{i,l,p \, \colon \, u_{i}^{l}(p) = y} \left\langle \vec{\eta}(y), \frac{Du_{i}^{l}(p)_{\sharp}\vec{\tau}(p)}{|Du_{i}^{l}(p)_{\sharp}\vec{\tau}(p)|}\right\rangle.
\end{equation}
\end{itemize}
\end{proposition}

\begin{remark} \label{ext:lip_man}
The definition of push-forward can be easily extended to the case when the domain $\Sigma$ is a Lipschitz oriented $m$-dimensional submanifold. In this case, indeed, there are countably many submanifolds $\Sigma_{j}$ of class $C^{1}$ which cover $\Ha^{m}$-a.a. $\Sigma$, and such that the orientations of $\Sigma_{j}$ and $\Sigma$ coincide on their intersection (see \cite[Theorem 5.3]{Sim83}). Hence, if $B \subset \Sigma$ is a measurable subset and $u \colon B \to \A_{Q}(\R^{n})$ is Lipschitz and proper, then the push-forward of $\llbracket B \rrbracket$ through $u$ can be defined to be the integer rectifiable current ${\bf T}_{u} := \sum_{j=1}^{\infty} {\bf T}_{u_{j}}$, where $u_{j} := u|_{B \cap \Sigma_{j}}$. All the conclusions of Proposition \ref{Q_pf:repr} remain valid in this context (cf. \cite[Lemma 1.7]{DLS13a}). Furthermore, the push-forward is invariant with respect to bi-Lipschitz homeomorphisms: if $u \colon \Sigma \to \A_{Q}(\R^{n})$ is Lipschitz and proper, $\phi \colon \tilde{\Sigma} \to \Sigma$ is bi-Lipschitz and $\tilde{u} := u \circ \phi$, then ${\bf T}_{\tilde{u}} = {\bf T}_{u}$. 

In Appendix \ref{ssec:pf_fc} we will take advantage of the polyhedral approximation theorem in order to further extend the domain of the multi-valued push-forward operator to the much larger class of integral flat chains.
\end{remark}

The notion of push-forward allows one to associate a rectifiable current to the \emph{graph} of a multiple valued function. Here and in the sequel, if $\Sigma \subset \R^{d}$ is an $m$-dimensional Lipschitz submanifold and $u \colon B \subset \Sigma \to \A_{Q}(\R^{n})$ is a $Q$-valued map we will denote by ${\rm Gr}(u)$ the \emph{set-theoretical graph} of $u$, given by
\[
{\rm Gr}(u) := \lbrace \left( p,v \right) \in \R^{d} \times \R^{n} \, \colon \, p \in B, \, v \in \spt(u(p)) \rbrace.
\]
\begin{definition}\label{def_graph}
Let $u = \sum_{l} \llbracket u_{l} \rrbracket \colon B \subset \Sigma \to \A_{Q}(\R^{n})$ be a proper Lipschitz $Q$-valued map, and define the map 
\[
{\rm Id} \times u \colon p \in B \mapsto \sum_{l=1}^{Q} \llbracket \left( p, u_{l}(p) \right) \rrbracket \in \A_{Q}(\R^{d} \times \R^{n}).
\]
Then, the push-forward ${\bf T}_{{\rm Id} \times u}$ is the integer rectifiable current associated to ${\rm Gr}(u)$, and will be denoted by ${\bf G}_{u}$.
\end{definition}

An important feature of the notion of push-forward of Lipschitz manifolds through multiple valued functions is that, exactly as in the single valued context, it behaves nicely with respect to the boundary operator.

\begin{theorem}[Boundary of the push-forward] \label{pf_bdry:thm}
Let $\Sigma \subset \R^{d}$ be an $m$-dimensional Lipschitz manifold with Lipschitz boundary, and let $u \colon \Sigma \to \A_{Q}(\R^{n})$ be a proper Lipschitz map. Then, $\partial {\bf T}_{u} = {\bf T}_{u|_{\partial \Sigma}}$. 
\end{theorem}

The first instance of such a result appears already in \cite[Section 1.6]{Almgren00}, where Almgren relies on the intersection theory of flat chains to define a multi-valued push-forward operator acting on flat chains and study its properties. A more elementary proof was then suggested by De Lellis and Spadaro in \cite[Theorem 2.1]{DLS13a}. In Appendix \ref{sec:commutation}, we will provide a slightly simplified version of their proof, relying on a double inductive process, both on the number $Q$ of values that the function takes and on the dimension $m$ of the domain.

\section{$Q$-multisections} \label{sec:sections}

The goal of this section is to define the notion of \emph{multiple valued section} of an abstract vector bundle over a given Riemannian base manifold. The main ideas of this section were introduced in the unpublished note \cite{Allard}, where Allard studies the properties of the push-forward of the elements of a fairly large subclass of the class of integer rectifiable currents on a given manifold through coherent and vertically limited $Q$-valued sections of a vector bundle on the manifold. This is more than what we need to prove the reparametrization theorem of Section \ref{sec:reparam}, for which we will instead only use the elementary theory of \S\, \ref{sec:pf} and the new techniques discussed in the coming paragraphs.

\subsection{Preliminary definitions.} In what follows, $\Sigma = \Sigma^{m}$ denotes an $m$-dimensional Riemannian manifold of class $C^{1}$, and $E$ is an $(m+n)$-dimensional manifold which is the total space of a vector bundle $\Pi \colon E \to \Sigma$ of rank $n$ and class $C^{1}$ over the base manifold $\Sigma$. Following standard notations, we will denote by $E_{p} := \Pi^{-1}(\{p\})$ the \emph{fiber} over the base point $p \in \Sigma$. We will let $\{(\U_{\alpha}, \Psi_{\alpha})\}_{\alpha \in I}$ be a locally finite family of \emph{local trivializations} of the bundle: thus, $\{\U_{\alpha}\}$ is a locally finite open covering of the manifold $\Sigma$, and 
\[
\Psi_{\alpha} \colon \Pi^{-1}(\U_{\alpha}) \to \U_{\alpha} \times \R^{n}
\]
are differentiable maps satisfying:

\begin{itemize}
\item[$(i)$] $\p_{1} \circ \Psi_{\alpha} = \Pi|_{\Pi^{-1}(\U_{\alpha})}$, where $\p_{1} \colon \U_{\alpha} \times \R^n \to \U_{\alpha}$ is the projection on the first factor;

\item[$(ii)$] for any $\alpha, \beta \in I$ with $\U_{\alpha} \cap \U_{\beta} \neq \emptyset$, there exists a differentiable map
\[
\tau_{\alpha\beta} \colon \U_{\alpha} \cap \U_{\beta} \to {\rm GL}(n, \R)
\]
with the property that
\[
\Psi_{\alpha} \circ \Psi_{\beta}^{-1}(p,v) = (p, \tau_{\alpha\beta}(p) \cdot v) \hspace{1cm} \forall \, p \in \U_{\alpha} \cap \U_{\beta}, \, \forall \, v \in \R^{n}.
\]
\end{itemize}

Without loss of generality, we can assume that each open set $\U_{\alpha}$ is also the domain of a local chart $\psi_{\alpha} \colon \U_{\alpha} \to \R^{m}$ on $\Sigma$.

Let now $Q$ be an integer, $Q \geq 1$. 


\begin{definition}[$Q$-valued sections, Allard {\cite{Allard}}] \label{multi:def}
Given a vector bundle $\Pi \colon E \to \Sigma$ as above, and a subset $B \subset \Sigma$, a $Q$-multisection over $B$ is a map
\begin{equation} \label{multi:eq1}
M \colon E \to \N
\end{equation}
with the property that
\begin{equation} \label{multi:eq2}
\sum_{\xi \in E_{p}} M(\xi) = Q \hspace{1cm} \mbox{for every } p \in B.
\end{equation}
\end{definition}

\begin{remark}
If $s \colon B \to E$ is a classical local section, then the map $M \colon E \to \N$ defined by
\begin{equation} \label{1sect}
M(\xi) :=
\begin{cases}
1, & \mbox{if there exists } p \in B \mbox{ such that } \xi = s(p), \\
0, & \mbox{otherwise}
\end{cases}
\end{equation}
is evidently a $1$-multisection over $B$, according to Definition \ref{multi:def}. On the other hand, given a $1$-multisection $M$, condition \eqref{multi:eq2} ensures that for every $p \in B$ there exists a unique $\xi \in E_{p}$ such that $M(\xi) > 0$. If such an element $\xi$ is denoted $s(p)$, then the map $p \mapsto s(p)$ defines a classical section of the bundle $E$ over $B$. Hence, $1$-multisections over a subset $B$ are just (possibly rough) sections over $B$ in the classical sense.
\end{remark}

The above Remark justifies the name that was adopted for the objects introduced in Definition \ref{multi:def}: $Q$-multisections are simply the $Q$-valued counterpart of classical sections of a vector bundle. From a different point of view, we may say that $Q$-multisections generalize Almgren's $Q$-valued functions to vector bundle targets. Indeed, $Q$-valued functions defined on a manifold $\Sigma$ might be seen as $Q$-multisections of a trivial bundle over $\Sigma$, as specified in the following remark.

\begin{remark}
Assume $E$ is the trivial bundle of rank $n$ over $\Sigma$, that is $E = \Sigma \times \R^{n}$ and $\Pi$ is the projection on the first factor. Then, to any $Q$-multisection $M$ over $\Sigma$ it is possible to associate the multiple valued function $u_{M} \colon \Sigma \to \A_{Q}(\R^n)$ defined by 
\begin{equation} \label{Qfunct}
u_{M}(p) := \sum_{v \in \R^{n}} M(p,v) \llbracket v \rrbracket.
\end{equation}
Conversely, if $u \colon \Sigma \to \A_{Q}(\R^{n})$ is a multiple valued function then one can define the $Q$-multisection $M_{u}$ induced by $u$ simply setting
\begin{equation} \label{induced_multi}
M_{u}(p,v) := \Theta_{u(p)}(v), 
\end{equation}
where $\Theta_{u(p)}(v)$ is the multiplicity of the vector $v$ in $u(p)$.
\end{remark}

\subsection{Coherent and vertically limited multisections.} 

\begin{definition}[Coherence, Allard {\cite{Allard}}] \label{Coherence}
A $Q$-multisection $M$ of the vector bundle $\Pi \colon E \to \Sigma$ over $\Sigma$ is said to be \emph{coherent} if the following holds. For every $p \in \Sigma$ and for every disjoint family $\mathcal{V}$ of open sets $V \subset E$ such that each member $V \in \mathcal{V}$ contains exactly one element of ${M_{p} := \lbrace \xi \in E_{p} \, \colon \, M(\xi) > 0 \rbrace}$, there is an open neighborhood $U$ of $p$ in $\Sigma$ such that for any $q \in U$
\begin{equation} \label{Coherence:eq}
\sum_{\zeta \in M_{q} \cap V} M(\zeta) = M(\xi) \hspace{1cm} \mbox{if } \xi \in M_{p} \cap V .
\end{equation} 
\end{definition}

The following proposition motivates the necessity of introducing the notion of coherence: it is a way of generalizing the continuity of $Q$-valued functions in the vector bundle-valued context.

\begin{proposition} \label{continuity:prop}
Let $E = \Sigma \times \R^{n}$ be the trivial bundle of rank $n$ over $\Sigma$. Then, a $Q$-multisection $M$ is coherent if and only if the associated multiple valued function ${u_{M} \colon \Sigma \to \A_{Q}(\R^{n})}$ is continuous.
\end{proposition}

\begin{proof}
Let $u \colon \Sigma \to \A_{Q}(\R^n)$ be a continuous $Q$-valued function, and let $M \colon \Sigma \times \R^{n} \to \N$ be the induced multisection defined by \eqref{induced_multi}. In order to show that $M$ is coherent, fix a point $p$ in the base manifold $\Sigma$, and decompose $u(p) = \sum_{j=1}^{J} m_{j} \llbracket v_{j} \rrbracket$ so that $v_{j} \neq v_{j'}$ when $j \neq j'$ and $m_{j} := M(p,v_j)$. Now, let $\mathcal{V} = \{V_{1}, \dots, V_{J} \}$ be a disjoint family of open sets $V_{j} \subset \R^{n}$ with the property that if $M_{p} := \lbrace v \in \R^n \, \colon \, M(p,v) > 0 \rbrace$ then $M_{p} \cap V_{j} = \{v_{j}\}$. Let $\varepsilon > 0$ be a radius such that $B_{\varepsilon}(v_j) \subset V_j$ for every $j = 1,\dots,J$. Then, since $u$ is continuous, there exists a neighborhood $U$ of $p$ in $\Sigma$ such that 
\[
u(q) \in \B_{\frac{\varepsilon}{2}}(u(p)) := \left\lbrace T \in \A_{Q}(\R^n) \, \colon \, \G(T, u(p)) < \frac{\varepsilon}{2} \right\rbrace, 
\]
for every $q \in U$. From the definition of the metric $\G(\cdot, \cdot)$ in $\A_{Q}(\R^{n})$, it follows naturally that for every $q \in U$ it has to be
\[
\sum_{w \in V_{j}} M(q,w) = m_{j} \hspace{1cm} \mbox{for every } j \in \{1,\dots,J\},
\]
and thus $M$ is coherent.

Conversely, suppose $M$ is a coherent $Q$-multisection of the trivial bundle $\Sigma \times \R^{n}$, and let $u$ be the associated multiple valued funtion as defined in \eqref{Qfunct}. The goal is to prove that $u$ is continuous. Fix any point $p \in \Sigma$, and let $\{ p_{h} \}_{h=1}^{\infty} \subset \Sigma$ be any sequence such that $p_{h} \to p$. Since $M$ is coherent, for any ball $B_{R} \subset \R^{n}$ such that $\spt(u(p)) \subset B_{R}$ there exists $h_{0} \in \N$ such that $\spt(u(p_{h})) \subset B_{R}$ for every $h \geq h_{0}$. In particular, $|u(p_{h})|^{2} := \G(u(p_{h}), Q \llbracket 0 \rrbracket)^{2} \leq Q R^{2}$ for every $h \geq h_{0}$, and thus the measures $\{ u(p_{h}) \}$ have uniformly finite second moment. Therefore, since the metric $\G$ on $\A_{Q}(\R^{n})$ coincides with the $L^{2}$-based Wasserstein distance on the space of positive measures with finite second moment, from \cite[Proposition 7.1.5]{AGS08} immediately follows that $\G(u(p_{h}), u(p)) \to 0$ if and only if the sequence $u(p_{h})$ \emph{narrowly converges} to $u(p)$, that is if and only if
\begin{equation} \label{continuity:1}
\lim_{h \to \infty} \langle u(p_{h}), f \rangle = \langle u(p), f \rangle \hspace{1cm} \forall \, f \in C_{b}(\R^{n}),
\end{equation}
that is, explicitly,
\begin{equation} \label{continuity:2}
\lim_{h \to \infty} \sum_{v \in \R^{n}} M(p_{h}, v) f(v) = \sum_{v \in \R^{n}} M(p, v) f(v) \hspace{1cm} \forall \, f \in C_{b}(\R^{n}),
\end{equation}
where $C_{b}(\R^n)$ denotes the space of bounded continuous functions on $\R^n$. So, in order to prove this, fix $f \in C_{b}(\R^{n})$ and $\varepsilon > 0$. Let $v_{1}, \dots, v_{J}$ be distinct points in $M_{p}$, and let $\eta = \eta(\varepsilon) > 0$ be a number chosen in such a way that
\begin{equation} \label{continuity:3}
|v - v_{j}| < \eta \implies |f(v) - f(v_{j})| < \varepsilon \hspace{1cm} \mbox{for every } j=1,\dots,J.
\end{equation}
Choose now radii $r_{1}, \dots, r_{J}$ such that $r_{j} < \frac{\eta}{2}$, the balls $B_{j} := B_{r_{j}}(v_{j})$ are pairwise disjoint and $M(p,v) = 0$ for any $v \in B_{j} \setminus \{v_{j}\}$. Since $M$ is coherent, in correspondence with the choice of the family $\{B_{j}\}$ there is an open neighborhood $U$ of $p$ in $\Sigma$ with the property that 
\begin{equation} \label{continuity:4}
\sum_{v \in B_{j}} M(q,v) = M(p,v_{j}) \hspace{1cm} \mbox{for every } q \in U.
\end{equation}
Since $\sum_{j=1}^{J} M(p,v_{j}) = Q$, equation \eqref{continuity:4} implies that
\begin{equation} \label{continuity:5}
\sum_{j=1}^{J} \sum_{v \in B_{j}} M(q,v) = Q \hspace{1cm} \mbox{for every } q \in U,
\end{equation}
and thus, whenever $q \in U$, $M(q,v) = 0$ if $v \notin \bigcup_{j=1}^{J} B_{j}$. Therefore, only the balls $B_{j}$ are relevant, namely
\begin{equation} \label{continuity:6}
\sum_{v \in \R^{n}} M(q,v) = \sum_{j=1}^{J} \sum_{v \in B_{j}} M(q,v) \hspace{1cm} \mbox{for every } q \in U.
\end{equation}
We can now finally conclude the validity of \eqref{continuity:2}: Let $N \in \N$ be such that $p_{h} \in U$ for every $h \geq N$ and estimate, for such $h$'s:
\[
\begin{split}
\left| \sum_{v \in \R^{n}} M(p_{h},v) f(v) - \sum_{v \in \R^{n}} M(p,v) f(v) \right| &\overset{\eqref{continuity:6}}{=} \left| \sum_{j=1}^{J} \sum_{v \in B_{j}} M(p_{h},v) f(v) - \sum_{j=1}^{J} M(p,v_{j}) f(v_{j}) \right| \\
&\leq \sum_{j=1}^{J} \left| \sum_{v \in B_{j}} M(p_h, v) f(v) - M(p,v_j) f(v_j) \right| \\
&\overset{\eqref{continuity:4}}{\leq} \sum_{j=1}^{J} \sum_{v \in B_{j}} M(p_{h},v) |f(v) - f(v_{j})| \\
&\overset{\eqref{continuity:5}}{\leq} Q \varepsilon,
\end{split}
\]
which completes the proof.
\end{proof}

The next step will be to define a suitable property of $Q$-multisections that is equivalent to Lipschitz continuity of the associated multiple valued function whenever such an association is possible. We start from a definition in the easy case when the vector bundle $E$ coincides with the trivial bundle $\Omega \times \R^{n}$ over an open subset $\Omega \subset \R^{m}$. 

\begin{definition}[$\tau$-cone condition, Allard {\cite{Allard}}] \label{cone}
Let $\tau > 0$ be a real number. We say that a $Q$-multisection $M \colon \Omega \times \R^{n} \to \N$ satisfies the $\tau$\emph{-cone condition} if the following holds. For any $x \in \Omega$, for any $v \in M_{x} = \{ v \in \R^{n} \, \colon \, M(x,v) > 0 \}$, there exist neighborhoods $U$ of $x$ in $\Omega$ and $V$ of $v$ in $\R^{n}$ such that
\begin{equation} \label{cone:eq}
\{ (y,w) \in U \times V \, \colon \, M(y,w) > 0 \} \subset \K^{\tau}_{x,v},
\end{equation}
where
\begin{equation} \label{cone:def}
\K^{\tau}_{x,v} := \{ (y,w) \in \R^{m} \times \R^{n} \, \colon \, |w - v| \leq \tau |y - x| \}
\end{equation}
is the $\tau$-cone centered at $(x,v)$ in $\R^{m} \times \R^{n}$. 
\end{definition}

\begin{proposition} \label{Lip:prop}
Let $\Omega \subset \R^{m}$ be open and convex. If $u$ is an $\ell$-Lipschitz $Q$-valued function, then the induced multisection $M_{u} \colon \Omega \times \R^{n} \to \N$ is coherent and satisfies a $\tau$-cone condition with $\tau = \ell$. Conversely, if a $Q$-multisection of the bundle $\Omega \times \R^{n}$ is coherent and satisfies the $\tau$-cone condition, then the associated $Q$-valued function $u_{M} \colon \Omega \to \A_{Q}(\R^{n})$ is Lipschitz with $\Lip(u_{M}) \leq \sqrt{Q} \tau$. 
\end{proposition}

\begin{proof}
One implication is immediate. Indeed, first observe that the continuity of $u$ implies that $M = M_{u}$ is coherent by Proposition \ref{continuity:prop}. Then, fix $x \in \Omega$, and suppose that $u(x) = \sum_{j=1}^{J} m_{j} \llbracket \tilde{v}_{j} \rrbracket$, with the $\tilde{v}_{j}$'s all distinct and $m_{j} := M(x,\tilde{v}_{j})$. Let $\varepsilon > 0$ be such that the balls $B_{\varepsilon}(\tilde{v}_{j}) \subset \R^{n}$ are a disjoint family of open sets such that $M_{x} \cap B_{\varepsilon}(\tilde{v}_{j}) = \{\tilde{v}_{j}\}$. Since $M$ is coherent, there exists an open neighborhood $U$ of $x$ in $\Omega$ such that the following two properties are satisfied for any $y \in U$:
\begin{itemize}
\item[$(i)$] $\sum_{w \in B_{\varepsilon}(\tilde{v}_{j})} M(y,w) = m_{j}$;
\item[$(ii)$] if $u(x) = \sum_{l=1}^{Q} \llbracket v_{l} \rrbracket$ with the first $m_{1}$ of the $v_{l}$'s all identically equal to $\tilde{v}_{1}$, the next $m_{2}$ all identically equal to $\tilde{v}_{2}$ and so on, and if $u(y) = \sum_{l=1}^{Q} \llbracket w_{l} \rrbracket$ with the $w_{l}$ (not necessarily all distinct) ordered in such a way that $\G(u(x), u(y)) = \left( \sum_{l=1}^{Q} |v_{l} - w_{l}|^{2} \right)^{\frac{1}{2}}$, then $w_{l} \in B_{\varepsilon}(v_{l})$ for every $l \in \{1,\dots,Q\}$. 
\end{itemize} 
Thus, for such $y$'s it is evident that the Lipschitz condition $\G(u(y), u(x)) \leq \ell |y - x|$ forces $|w_{l} - v_{l}| \leq \ell |y - x|$ for every $l = 1,\dots,Q$, which is to say that for every $j = 1,\dots,J$
\[
\{ (y,w) \in U \times B_{\varepsilon}(\tilde{v}_{j}) \, \colon \, M(y,w) > 0 \} \subset \K^{\ell}_{x,\tilde{v}_{j}},
\] 
as we wanted.

For the converse, consider a $Q$-multisection $M$, and assume it is coherent and satisfies the $\tau$-cone condition. Define $u \colon \Omega \to \A_{Q}(\R^{n})$ as in \eqref{Qfunct}. We will first prove the following claim, from which the Lipschitz continuity of $u$ will easily follow:

\hspace*{0.5cm} \textbf{Claim.} For every $x \in \Omega$ there exists an open neighborhood $U_{x}$ of $x$ in $\Omega$ such that
\begin{equation} \label{claim}
\G(u(y), u(x)) \leq \sqrt{Q} \tau |y - x| \hspace{1cm} \mbox{for every } y \in U_{x}.
\end{equation} 

In order to show this, fix a point $x \in \Omega$, and let $\{v_{1}, \dots, v_{J} \}$ be distinct vectors in $M_{x}$. Since $M$ satisfies the $\tau$-cone condition, there exist open neighborhoods $U$ of $x$ in $\Omega$ and $V_{j}$ of $v_{j}$ in $\R^{n}$ for every $j = 1,\dots,J$ such that
\begin{equation} \label{Lip:1}
\lbrace (y,w) \in U \times V_{j} \, \colon \, M(y,w) > 0 \rbrace \subset \K^{\tau}_{x,v_{j}} \hspace{1cm} \forall \, j=1,\dots,J.
\end{equation} 
In particular, condition \eqref{Lip:1} implies that $M_{x} \cap V_{j} = \{v_{j}\}$ for every $j$. Up to shrinking the $V_{j}$'s if necessary, we can also assume that they are pairwise disjoint. Hence, since $M$ is also coherent, we can conclude the existence of a (possibly smaller) neighborhood of $x$, which we will still denote $U$, with the property that not only \eqref{Lip:1} is satisfied but also
\begin{equation} \label{Lip:2}
\sum_{w \in V_{j}} M(y,w) = M(x,v_{j}) \hspace{1cm} \forall \, j=1,\dots,J, \, \forall \, y \in U. 
\end{equation}
Therefore, if $y \in U$ we can write
\begin{equation} \label{Lip:3}
u(y) = \sum_{j=1}^{J} \sum_{w \in V_{j}} M(y,w) \llbracket w \rrbracket,
\end{equation}
whereas
\begin{equation} \label{Lip:4}
u(x) = \sum_{j=1}^{J} M(x,v_{j}) \llbracket v_{j} \rrbracket.
\end{equation}
Using \eqref{Lip:2}, \eqref{Lip:3}, \eqref{Lip:4} and the fact that if $(y,w) \in U \times V_{j}$ then $M(y,w) > 0 \implies |w - v_{j}| \leq \tau |y - x|$, we immediately conclude that
\begin{equation} \label{Lip:5}
\G(u(y), u(x))^{2} \leq \left( \sum_{j=1}^{J} M(x,v_{j}) \right) \tau^{2} |y - x|^{2} \hspace{1cm} \mbox{for every } y \in U, 
\end{equation}
which proves our claim.

Next, we prove that $u$ is Lipschitz continuous with $\Lip(u) \leq \sqrt{Q} \tau$. To achieve this, fix two distinct points $p,q \in \Omega$. Since $\Omega$ is convex, the segment $\left[ p, q \right]$ is contained in $\Omega$, and let $e$ denote the unit vector orienting the segment $\left[ p, q \right]$ in the direction from $p$ to $q$. By the claim, for every $x \in \left[ p, q \right]$ there exists a radius $r_{x} > 0$ such that
\begin{equation} \label{Lip:6}
\G(u(y), u(x)) \leq \sqrt{Q} \tau |y - x| \hspace{1cm} \mbox{for every } y \in I_{x} := \left( x - r_{x} e, x + r_{x} e \right).
\end{equation}
The open intervals $I_{x}$ are clearly an open covering of $\left[ p, q \right]$. Since the segment is compact, it admits a finite subcovering, which will be denoted $\{ I_{x_{i}} \}_{i=0}^{N}$. We may assume, refining the subcovering if necessary, that an interval $I_{x_i}$ is not completely contained in an interval $I_{x_j}$ if $i \neq j$. If we relabel the indices of the points $x_{i}$ in a non-decreasing order along the segment, we can now choose an auxiliary point $y_{i,i+1}$ in $I_{x_i} \cap I_{x_{i+1}} \cap \left( x_{i}, x_{i+1} \right)$ for each $i = 0,\dots,N-1$. We can finally conclude:
\begin{equation} \label{Lip:7}
\begin{split}
\G(u(p), u(q)) \leq & \, \G(u(p), u(x_{0})) \\
& + \sum_{i=0}^{N-1} \left( \G(u(x_{i}), u(y_{i,i+1})) + \G(u(y_{i,i+1}), u(x_{i+1})) \right) + \G(u(x_N), u(q)) \\
\overset{\eqref{Lip:6}}{\leq} & \sqrt{Q} \tau \left( |x_{0} - p| + \sum_{i=0}^{N-1} \left( |y_{i,i+1} - x_{i}| + |x_{i+1} - y_{i,i+1}| \right) + |q - x_{N}| \right) \\
= & \sqrt{Q} \tau |q - p|,
\end{split}
\end{equation} 
which completes the proof.
\end{proof}

\begin{definition}[Allard, {\cite{Allard}}] \label{vert_lim_def} 
Let $\Pi \colon E \to \Sigma$ be a vector bundle, $M$ a $Q$-multisection over $\Sigma$ and $\tau > 0$. We say that $M$ is $\tau$-\emph{vertically limited} if for any coordinate domain $\U_{\alpha}$ on $\Sigma$ with associated chart $\psi_{\alpha} \colon \U_{\alpha} \to \R^m$ and trivialization $\Psi_{\alpha} \colon \Pi^{-1}(\U_{\alpha}) \to \U_{\alpha} \times \R^n$ the multisection
\[
M_{\alpha} := M \circ \Psi_{\alpha}^{-1} \circ \left( \psi_{\alpha}^{-1} \times {\rm id}_{\R^n} \right) \colon \psi_{\alpha}(\U_{\alpha}) \times \R^n \to \N
\]
satisfies the $\tau$-cone condition.
\end{definition}

\begin{remark}
Note that the constant $\tau$ in Definition \ref{vert_lim_def} depends on the choice of the family of local trivializations of the vector bundle. 
\end{remark}

\section{Reparametrization of multiple valued graphs} \label{sec:reparam}

In the remaining part of this note, we will apply the theory of $Q$-multisections of a vector bundle in order to derive a more elementary proof of the reparametrization theorem for multiple valued graphs mentioned in the Introduction.

Before stating the precise result we are aiming at, we need to introduce some notation and terminology, which will be used throughout the whole section.
\begin{assumptions} \label{Ass}
Let $m$, $n$ and $Q$ denote fixed positive integers. Let also $0 < s < r <1$. We will consider the following:
\begin{itemize}
\item[$(A1)$] an open $m$-dimensional submanifold $\Sigma$ of the Euclidean space $\R^{m+n}$ with ${\Ha^{m}(\Sigma) < \infty}$ which is the graph of a function $\bphi \colon B_{s} \subset \R^{m} \to \R^{n}$ with $\| \bphi \|_{C^3} \leq \bar{c}$;
\item[$(A2)$] a regular tubular neighborhood ${\bf U}$ of $\Sigma$, that is the set of points
\begin{equation} \label{tub_neigh}
{\bf U} := \lbrace \xi = p + {\rm v} \, \colon \, p \in \Sigma, \, {\rm v} \in T_{p}^{\perp}\Sigma, \, |{\rm v}| < c_{0} \rbrace,
\end{equation}
where the thickness $c_{0}$ is small enough to guarantee that the nearest point projection $\Pi \colon {\bf U} \to \Sigma$ is well defined and $C^2$;
\item[$(A3)$] a proper Lipschitz $Q$-valued function $f \colon B_{r} \subset \R^{m} \to \A_{Q}(\R^n)$.
\end{itemize}
\end{assumptions}
Some comments about the objects introduced in Assumptions \ref{Ass} are now in order. First observe that the map $\bphi$ induces a parametrization of the manifold $\Sigma$, which we denote by
\begin{equation} \label{parametrization}
{\bf \Phi} \colon x \in B_{s} \subset \R^{m} \mapsto {\bf \Phi}(x) := (x, \bphi(x)) \in \R^{m+n}.
\end{equation}
The inverse of ${\bf \Phi}$ can be used as a global chart on $\Sigma$. If $p \in \Sigma$, then $\pi_{p}$ and $\varkappa_{p}$ will denote the tangent space $T_{p}\Sigma$ and its orthogonal complement in $\R^{m+n}$ respectively. The symbols $\pi_{0}$ and $\pi_{0}^{\perp}$, instead, will be reserved for the planes $\R^{m} \times \{0\} \simeq \R^{m}$ and $\{0\} \times \R^{n} \simeq \R^{n}$ respectively. In general, if $\pi$ is a linear subspace of $\R^{m+n}$, the symbol $\p_{\pi}$ will denote orthogonal projection onto it.   

Concerning the tubular neighborhood ${\bf U}$, we will denote by $\lbrace \nu_{1}, \dots, \nu_{n} \rbrace$ the standard orthonormal frame of the normal bundle of $\Sigma$ described in \cite[Appendix A]{DLS13a}. Such a frame is simply obtained by applying, at every point $p \in \Sigma$, the Gram-Schmidt orthogonalization algorithm to the vectors $\p_{\varkappa_p}(e_{m+1}), \dots, \p_{\varkappa_p}(e_{m+n})$, where $\lbrace e_{m+1}, \dots, e_{m+n} \rbrace$ is the standard orthonormal basis of $\{0\} \times \R^{n} \subset \R^{m+n}$. The analytic properties of the frame $\nu_{1}, \dots \nu_{n}$ are recorded in the following lemma.

\begin{lemma}[cf. {\cite[Lemma A.1]{DLS13a}}] \label{trivialization}
If $\| D \bphi \|_{C^0}$ is smaller than a geometric constant, then $\nu_{1}, \dots \nu_{n}$ is an orthonormal frame spanning $\varkappa_{p}$ at every $p \in \Sigma$. Consider $\nu_{i}$ as functions of $x \in B_{s}$ using the inverse of ${\bf \Phi}$ as a chart. For every $\gamma + k \geq 0$, there is a constant $C = C(m,n,\gamma,k)$ such that if $\| \bphi \|_{C^{k+2, \gamma}} \leq 1$, then $\| D \nu_{i} \|_{C^{k,\gamma}} \leq C \| D \bphi \|_{C^{k+1,\gamma}}$. 
\end{lemma}

Recall that, for any $Q$-valued function $f$ as in assumption $(A3)$, ${\rm Gr}(f)$ and ${\bf G}_{f}$ denote the set-theoretical graph of $f$ and the integral $m$-current associated to it respectively. The concept of \emph{reparametrization} of $f$ is introduced next.

\begin{definition} \label{reparametrization:def}
Given $\Sigma$, ${\bf U}$ and $f$ as in Assumptions \ref{Ass}, we call a \emph{Lipschitz normal reparametrization} of the $Q$-function $f$ in the tubular neighborhood ${\bf U}$ any $Q$-valued function $F \colon \Sigma \to \A_{Q}(\R^{m+n})$ such that the following conditions are satisfied:
\begin{itemize}
\item[$(i)$] for every $p \in \Sigma$, $F(p) = \sum_{l=1}^{Q} \llbracket p + N_{l}(p) \rrbracket$, with $N \colon \Sigma \to \A_{Q}(\R^{m+n})$ a Lipschitz continuous $Q$-valued function;
\item[$(ii)$] $p + N_{l}(p) \in {\bf U}$ and $N_{l}(p) \in \varkappa_{p} = T_{p}^{\perp}\Sigma$ for every $l \in \{1,\dots,Q\}$, for every $p \in \Sigma$;
\item[$(iii)$] ${\bf T}_{F} = {\bf G}_{f} \mres {\bf U}$.
\end{itemize} 
\end{definition}

We are now ready to state the main theorem of this section.

\begin{samepage}
\begin{theorem}[Existence of the reparametrization] \label{reparametrization:thm}
Let $Q$, $m$ and $n$ be positive integers, and $0 < s < r < 1$. Then, there are constants $c_{0}, C > 0$ (depending on $m$, $n$, $Q$, $r - s$ and $\frac{r}{s}$) with the following property. For any $\bphi$, $\Sigma$, ${\bf U}$ and $f$ as in Assumptions \ref{Ass} such that
\begin{equation} \label{reparametrization:hp}
\| \bphi \|_{C^{2}} + \Lip(f) \leq c_{0}, \hspace{0.5cm} \| \bphi \|_{C^{0}} + \| f \|_{C^{0}} \leq c_{0} s,
\end{equation}
there exists a Lipschitz normal reparametrization $F$ of the $Q$-valued function $f$ in ${\bf U}$. Furthermore, the associated normal multi-valued vector field $N$ satisfies:
\begin{equation} \label{reparametrization:th1}
\Lip(N) \leq C \left( \| N \|_{C^0} \| D^{2}\bphi \|_{C^0} + \| D\bphi \|_{C^0} + \Lip(f) \right),
\end{equation}
\begin{equation} \label{reparametrization:th2}
\frac{1}{2\sqrt{Q}} |N({\bf \Phi}(x))| \leq \G(f(x), Q \llbracket \bphi(x) \rrbracket) \leq 2 \sqrt{Q} |N({\bf \Phi}(x))| \hspace{0.2cm} \forall \, x \in B_{s},
\end{equation}
\begin{equation} \label{reparametrization:th3}
|\bfeta \circ N({\bf \Phi}(x))| \leq C |\bfeta \circ f (x) - \bphi(x)| + C \Lip(f) |D\bphi(x)| |N({\bf \Phi}(x))| \hspace{0.2cm} \forall \, x \in B_s.
\end{equation}
Finally, assume $x \in B_s$ and $\left( x, \bfeta \circ f(x) \right) = p + {\rm v}$ for some $p \in \Sigma$ and ${\rm v} \in T_{p}^{\perp}\Sigma$. Then,
\begin{equation} \label{reparametrization:th4}
\G(N(p), Q \llbracket {\rm v} \rrbracket) \leq 2 \sqrt{Q} \G(f(x), Q \llbracket \bfeta \circ f(x) \rrbracket).
\end{equation}
\end{theorem}
\end{samepage}

As already mentioned in the Introduction, Theorem \ref{reparametrization:thm} was already proved in the same form by C. De Lellis and E. Spadaro in \cite[Theorem 5.1]{DLS13a}, and it plays an important role in the approximation of an area minimizing current at an interior branch point of density $Q$ with the graph of an almost ${\rm Dir}$-minimizing $Q$-valued function (cf. \cite[Theorem 2.4]{DLS13b}), which is the key for deducing the celebrated partial regularity result for area minimizing currents in codimension higher than one \cite[Theorem 0.3]{DLS13c}.


The argument will be divided into two parts: in the first part, we will suppose to be given $\Sigma$, ${\bf U}$ and $f$ as in Assumptions \ref{Ass}, and we will associate in an extremely natural way to the $Q$-valued function $f$ a $Q$-multisection $M$ of the tubular neighborhood ${\bf U}$, regarded as (the diffeomorphic image of) an open subset of a rank $n$ vector bundle of class $C^2$ over $\Sigma$. Under suitable smallness assumptions on the universal constant $c_{0}$ which controls the relevant norms of the functions $\bphi$ and $f$ as in \eqref{reparametrization:hp}, we will be able to show that the multisection $M$ so defined enjoys good properties of coherence and vertical boundedness. In the second part of the argument, we will produce the reparametrization $F$ using the multisection $M$ previously analyzed, and we will prove that the aforementioned geometric properties of $M$ do suffice to conclude the proof of Theorem \ref{reparametrization:thm}, using techniques that have been already introduced in the proofs of Propositions \ref{continuity:prop} and \ref{Lip:prop}.

We start with the first part of our program. Assume, therefore, that the manifold $\Sigma$, the tubular neighborhood ${\bf U}$ and the $Q$-valued function $f$ are given, and that the functions $\bphi$ and $f$ satisfy the bounds in \eqref{reparametrization:hp}. Suitable restrictions on the size of the constant $c_{0}$ will appear throughout the argument. Let
\begin{equation} \label{multi_f}
M_{f} \colon \R^{m} \times \R^{n} \to \N
\end{equation}
be the $Q$-multisection over $B_{r}$ associated to $f$. Observe that, setting $\ell := \Lip(f)$, Proposition \ref{Lip:prop} guarantees that $M_{f}$ is coherent and satisfies an $\ell$-cone condition. 

Now, we define a $Q$-multisection $M$ of the tubular neighborhood ${\bf U}$ as follows: for any $\xi \in {\bf U}$, $M(\xi)$ coincides with the multiplicity of the ``vertical coordinate'' $\p_{\pi_{0}^{\perp}}(\xi)$ in $f(\p_{\pi_{0}}(\xi))$. In symbols, we set:
\begin{equation} \label{reparametrized_multi}
M(\xi) := \Theta_{f(\p_{\pi_0}(\xi))}(\p_{\pi_{0}^{\perp}}(\xi)) = M_{f}(\p_{\pi_0}(\xi), \p_{\pi_0^\perp}(\xi)), \hspace{0.5cm} \mbox{for every } \xi \in {\bf U}.  
\end{equation}

The following Proposition shows that, under suitable smallness conditions on $c_{0}$, $M$ is indeed a coherent $Q$-multisection over the base manifold $\Sigma$.

\begin{proposition} \label{well-posedness}
If $c_{0}$ is small enough, depending on $m$, $n$, $r - s$ and $\frac{r}{s}$, then the identity
\begin{equation} \label{well-posedness:eq}
\sum_{\xi \in \Pi^{-1}(\{p\})} M(\xi) = Q
\end{equation}
holds for every $p \in \Sigma$, and thus $M$ is a $Q$-multisection over $\Sigma$. Moreover, $M$ is coherent.
\end{proposition}

\begin{proof}
First, we claim the following: the current $T := {\bf G}_{f} \mres {\bf U}$ satisfies $\Pi_{\sharp}T = Q \llbracket \Sigma \rrbracket$. In order to see this, fix a point $\xi \in \spt(T)$. By definition, $\xi = \left( y, f_{l}(y) \right)$ for some $y \in B_{r}$ and for some $l \in \{1,\dots,Q\}$; furthermore, there exist a point $p = (x, \bphi(x)) \in \Sigma$ and a vector ${\rm v} \in T_{p}^{\perp}\Sigma$ with $|{\rm v}| < c_{0}$ such that $\xi = p + {\rm v}$. Hence, we can easily estimate
\[
|y| = |\p_{\pi_{0}}(p + {\rm v})| \leq |x| + |{\rm v}| < s + c_{0}.
\]
This implies that if we choose $c_{0}$ suitably small, say
\begin{equation} \label{condition:1}
c_{0} \leq \frac{1}{2}(r - s),
\end{equation}
then the current $\left( \p_{\pi_{0}} \right)_{\sharp} T$ is compactly supported in $B_{r}$, and thus $(\partial {\bf G}_{f}) \mres {\bf U} = ({\bf G}_{f|_{\partial B_{r}}}) \mres {\bf U} = 0$. Now, we estimate more carefully the quantity $|{\rm v}| = |\xi - p| = \dist(\xi, \Sigma)$. Decompose 
\begin{equation} \label{wp:1}
|{\rm v}|^{2} = |\p_{\pi_{0}}({\rm v})|^{2} + |\p_{\pi_{0}^{\perp}}({\rm v})|^{2},
\end{equation}
and observe that the hypothesis \eqref{reparametrization:hp} readily implies that
\begin{equation} \label{wp:2}
|\p_{\pi_{0}^{\perp}}({\rm v})|^{2} = |f_{l}(y) - \bphi(x)|^{2} \leq c_{0}^{2} s^{2}.
\end{equation}
As for the ``horizontal'' component of the vector ${\rm v}$, write
\begin{equation} \label{wp:3}
{\rm v} = \sum_{i=1}^{n} v^{i} \nu_{i}(x),
\end{equation}
where $v = \left( v^{1}, \dots, v^{n} \right) \in \R^{n}$, $\lbrace \nu_{1}, \dots, \nu_{n} \rbrace$ is the standard orthonormal frame on the normal bundle of $\Sigma$ previously introduced, and where, with a slight abuse of notation, we are writing $\nu_{i}(x)$ instead of $\nu_{i}\left( {\bf \Phi}(x) \right)$. In this way,
\begin{equation} \label{wp:4}
|\p_{\pi_{0}}({\rm v})|^{2} \leq \left( \sum_{i=1}^{n} |v^{i}| |\p_{\pi_{0}}(\nu_{i}(x))| \right)^{2}.
\end{equation}
Clearly, in doing this we are tacitly assuming that $c_{0}$ is chosen so small that all the conclusions of Lemma \ref{trivialization} hold (in particular, we will always assume $c_{0} \leq 1$). Now, the quantity $|\p_{\pi_{0}}(\nu_{i}(x))|$ can be estimated by 
\begin{equation} \label{wp:5}
|\p_{\pi_{0}}(\nu_{i}(x))| \leq \left| \cos\left( \frac{\pi}{2} - \theta_{i}(x) \right) \right|,
\end{equation}
where $\theta_{i}(x)$ is the angle between $\nu_{i}(x)$ and $\p_{\pi_{0}^{\perp}}(\nu_{i}(x))$. In turn, this angle is controlled by $C |D\bphi(x)|$, with $C$ a geometric constant, because $\nu_{i}(x)$ is orthogonal to $T_{{\bf \Phi}(x)}\Sigma$. Thus, one has
\begin{equation} \label{wp:6}
|\p_{\pi_{0}}(\nu_{i}(x))| \leq \left| \sin(\theta_{i}(x)) \right| \leq C |D\bphi(x)|.
\end{equation}
Further estimating $|D\bphi(x)| \leq \|D\bphi\|_{C^0} \leq c_{0}$ by \eqref{reparametrization:hp} and inserting into \eqref{wp:4} yields:
\begin{equation} \label{wp:7}
|\p_{\pi_{0}}({\rm v})|^{2} \leq C c_{0}^{2} \left( \sum_{i=1}^{n} |v^{i}| \right)^{2} \leq C c_{0}^{2} \sum_{i=1}^{n} |v^{i}|^{2} = C c_{0}^{2} |{\rm v}|^{2},
\end{equation}
with $C = C(m,n)$. Combining \eqref{wp:1}, \eqref{wp:2} and \eqref{wp:7} produces 
\begin{equation} \label{wp:8}
|{\rm v}|^{2} \leq C c_{0}^{2} |{\rm v}|^{2} + c_{0}^{2} s^{2}.
\end{equation}
If
\begin{equation} \label{condition:2}
c_{0}^{2} \leq C^{-1} \left( 1 - \left( \frac{s}{r} \right)^{2} \right),
\end{equation}
then the term $C c_{0}^{2} |{\rm v}|^{2}$ on the right-hand side can be absorbed on the left-hand side, and in turn \eqref{wp:8} leads to
\begin{equation} \label{wp:9}
\dist(\xi, \Sigma)^{2} \leq c_{0}^{2} r^{2},
\end{equation}
which shows that the current $T$ is in fact compactly supported in ${\bf U}$. Together with the fact that ${\bf G}_{f}$ has no boundary in ${\bf U}$, such a result implies that the boundary of $T$ is actually supported in $\Pi^{-1}(\partial\Sigma)$ as soon as the constant $c_{0}$ is chosen in agreement with \eqref{condition:1} and \eqref{condition:2}. Hence, under these conditions we can deduce that $\partial \Pi_{\sharp} T$ is supported in $\partial \Sigma$. Thus, we are allowed to apply the constancy theorem (cf. \cite[Theorem 26.27]{Sim83}), and consequently conclude that $\Pi_{\sharp}T = k \llbracket \Sigma \rrbracket$ for some $k \in \mathbb{Z}$. In order to show that $k = Q$, we consider the functions $\bphi_{t} := t \bphi$ for $t \in \left[ 0, 1 \right]$, the corresponding manifolds $\Sigma_{t} := {\rm Gr}(\bphi_{t})$ with the associated projections $\Pi_{t} \colon {\bf U}_{t} \to \Sigma_{t}$. Also in this case, the constancy theorem produces $(\Pi_{t})_{\sharp}({\bf G}_{f} \mres {\bf U}_t ) = k(t) \llbracket \Sigma_{t} \rrbracket$. On the other hand, since the map
\[
t \in \left[ 0, 1 \right] \mapsto (\Pi_{t})_{\sharp}({\bf G}_{f} \mres {\bf U}_t)
\]
is continuous in the space of currents, one infers that $t \mapsto k(t)$ is a continuous integer-valued function, and thus it is constant. Since $k(0) = Q$, then necessarily also $k = k(1) = Q$, and the claim is proved.

Now, the fact that $\Pi_{\sharp}T = Q \llbracket \Sigma \rrbracket$ does not immediately imply that $\sum_{\xi \in M_p} M(\xi) = Q$, since there could in principle be cancellations, and the total mass on the fiber could be larger than $Q$. To see that this is not the case, consider, for every $p \in \Sigma$, the $0$-dimensional current $T_{p} := \langle {\bf G}_{f}, \Pi, p \rangle$ supported on the intersection ${\rm Gr}(f) \cap \Pi^{-1}(\{p\})$. By the slicing theory (cf. \cite[Section 4.3]{Federer69}), one has that there exists a set $Z \subset \Sigma$ with $\Ha^{m}(Z) = 0$ such that the following holds for every $p \in \Sigma \setminus Z$:
\begin{itemize}
\item[$(i)$] $T_{p}$ consists of a finite sum of Dirac masses $\sum_{j=1}^{J_p} m_{j} \llbracket \xi_{j} \rrbracket$ with coefficients $m_{j} \in \Z$;
\item[$(ii)$] for every $j \in \lbrace 1, \dots J_{p} \rbrace$, $\xi_{j} \in {\rm Gr}(f) \cap \Pi^{-1}(\{p\})$ and $|m_{j}| = M_{f}(\p_{\pi_{0}}(\xi_j), \p_{\pi_{0}^{\perp}}(\xi_j)) = M(\xi_j)$;
\item[$(iii)$] if $\vec{\nu}$ is the continuous unit $n$-vector orienting $\Pi^{-1}(\{p\})$ compatibly with the orientation of $\Sigma$, then the sign of $m_{j}$ is ${\rm sgn}\left( \langle \vec{T}(\xi_j) \wedge \vec{\nu}(\xi_j), \vec{e} \rangle \right)$, where $\vec{e} := e_{1} \wedge \dots \wedge e_{m} \wedge \dots \wedge e_{m+n}$, with $\{e_{1}, \dots, e_{m}\}$ the standard orthonormal basis of $\R^{m} \times \{0\}$ and $\{e_{m+1}, \dots, e_{m+n}\}$ the standard orthonormal basis of $\{0\} \times \R^{n}$.
\end{itemize}
Since $\| \bphi \|_{C^{1}} + \Lip(f) \leq c_{0}$, if $c_{0}$ is suitably small then every $\vec{T}(\xi_{j})$ is close to $\vec{e}_{m} := e_{1} \wedge \dots \wedge e_{m}$, whereas every $\vec{\nu}(\xi_j)$ is close to $\vec{e}_{n} := e_{m+1} \wedge \dots \wedge e_{m+n}$, and therefore every $m_{j}$ is positive. Since $\sum_{j=1}^{J_p} m_{j} = Q$ because $\Pi_{\sharp}T = Q \llbracket \Sigma \rrbracket$, we conclude that \eqref{well-posedness:eq} holds for every $p \in \Sigma \setminus Z$.

Therefore, if $\tilde{Z}$ denotes the set of points $p \in \Sigma$ such that \eqref{well-posedness:eq} does not hold (and hence $\sum_{\xi \in M_{p}} M(\xi) > Q$) then one has $\tilde{Z} \subset Z$. Now, we claim that in fact $\tilde{Z} = \emptyset$. At the same time, we will also prove that $M$ is coherent. In order to see this, fix $p \in \Sigma$, and assume that $M_{p} = \{ \xi_{1}, \dots, \xi_{J} \}$, with $m_{j} := M(\xi_{j})$. Let $\mathcal{V} = \{V_{1}, \dots V_{J} \}$ denote a collection of disjoint bounded open sets in ${\bf U}$ such that $M_{p} \cap V_{j} = \{\xi_{j}\}$ for every $j = 1,\dots,J$. We will show that for every $j$ there is an open neighborhood $\U_{j}$ of $p$ in $\Sigma$ such that
%

\begin{equation} \label{co_pr:1}
\sum_{\zeta \in M_{q} \cap V_{j}} M(\zeta) = m_{j} \hspace{0.5cm} \mbox{for every } q \in \U_{j},
\end{equation}  
so that the coherence condition will hold in $\U := \bigcap_{j} \U_{j}$. Consider the current $T_{j} := \Pi_{\sharp}({\bf G}_{f} \mres V_{j})$. We claim the following: there exists $\tilde{\U_{j}} \subset \Sigma$ open neighborhood of $p$ such that 
\begin{equation} \label{co_pr:2}
\spt\left( \partial(T_{j} \mres \tilde{\U_{j}})\right) \subset \partial \tilde{\U_{j}}.
\end{equation}
The validity of this claim allows us to conclude the proof, showing both that $\tilde Z = \emptyset$ and that $M$ is coherent. Let us assume the claim for the moment. If \eqref{co_pr:2} holds true, then the constancy theorem implies the existence of a constant $k_{j} \in \mathbb{N}$ such that $T_{j} \mres \tilde{\U_{j}} = k_{j} \llbracket \tilde{\U_j} \rrbracket$, which in turn gives $\Pi_{\sharp} \langle {\bf G}_f \mres V_j, \Pi, q \rangle = k_j  \llbracket q \rrbracket$ for $\Ha^m$-a.e. $q \in \tilde{\U_j}$, cf. for instance \cite[Theorem 4.3.2(7)]{Federer69}. Then, since no cancellations are allowed, $\M(\langle {\bf G}_f \mres V_j, \Pi, q \rangle ) = k_j$ for $\Ha^m$-a.e. $q \in \tilde{\U_j}$. At the same time, under the assumption that $c_0$ is smaller than a geometric constant, the Taylor expansion for the mass of a multi-valued Lipschitz graph (see \cite[Corollary 3.3]{DLS13a}) gives that
\begin{equation} \label{density estimate}
m_j - \frac{1}{2} \leq \liminf_{\rho \to 0} \frac{\|{\bf G}_f\|(B^{m+n}_{\rho}(\xi_j))}{\omega_m \rho^m} \leq \limsup_{\rho \to 0} \frac{\|{\bf G}_f\|(B^{m+n}_{\rho}(\xi_j))}{\omega_m \rho^m} \leq m_j + \frac{1}{2}\,.
\end{equation}
Thus, if we consider slices $\langle {\bf G}_f, \Pi, q \rangle$ at points $q$ in a suitable neighborhood $p \in \U_j \subset \tilde{\U_j}$ then the density estimate \eqref{density estimate} implies that $k_j \geq m_j$. Now, if, by contradiction, the point $p$ has been chosen in the set $\tilde Z$, namely $\sum_{j=1}^J m_j = Q' > Q$, then the above discussion gives that at $\Ha^m$-a.e. point $q \in \U = \bigcap_j \U_j$ one has $\sum_{\zeta \in M_q} M(\zeta) \geq Q' > Q$. Hence, $\tilde{Z}$ contains a set of full measure, which is the required contradiction. This concludes the proof that $M$ is a well-defined $Q$-multisection of the normal bundle of $\Sigma$. At the same time, the fact $\sum_{\zeta \in M_q} M(\zeta) = Q$ at \emph{every} $q$ forces $k_j = m_j$ for every $j$, and thus the validity of \eqref{co_pr:1} in a suitable neighborhood $p \in \U_j$ follows again from \eqref{density estimate}. This shows that $M$ is coherent.


Therefore, we are just left with showing \eqref{co_pr:2}. By contradiction, assume that there exists a sequence $\{p_{h}\}_{h=1}^{\infty} \subset \Sigma$ with $p_{h} \to p$ and such that $p_{h} \in \spt(\partial T_{j})$ for every $h$. Since the push-forward and boundary operators commute, and since ${\bf G}_{f}$ has no boundary in $V_{j}$, this would imply the existence of a sequence of points $\zeta_{h} \in {\rm Gr}(f) \cap \partial V_{j}$ such that $\Pi(\zeta_{h}) = p_{h}$. By the compactness of $\partial V_j$ and the continuity of the projection, a subsequence of the $\zeta_{h}$'s would converge to a point $\bar{\zeta} \in \partial V_j$ such that $\Pi(\bar{\zeta}) = p$. Furthermore, since $f$ is continuous ${\rm Gr}(f)$ is closed, and thus $\bar{\zeta} \in {\rm Gr}(f)$. But this is an evident contradiction, since by assumption ${\bf G}_{f}$ is supported outside of $\Pi^{-1}(\{p\}) \cap \partial V_j$. This shows the validity of \eqref{co_pr:2}, and concludes the proof of the Proposition.
\end{proof}

As an immediate consequence, the above result allows us to define the required reparametrization $F$: if $\Sigma$, ${\bf U}$ and $f$ are such that \eqref{reparametrization:hp} holds with the constant $c_{0}$ given by Proposition \ref{well-posedness}, we set
\begin{equation} \label{rep:Fdef}
F(p) := \sum_{\xi \in \Pi^{-1}(\{p\})} M(\xi) \llbracket \xi \rrbracket \hspace{0.5cm} \mbox{for every } p \in \Sigma.
\end{equation}  
By Proposition \ref{well-posedness}, $F$ is a well defined $Q$-valued function on $\Sigma$. By construction, the associated map $N \colon \Sigma \to \A_{Q}(\R^{m+n})$ given by
\begin{equation} \label{rep:Ndef}
N(p) := \sum_{\xi \in \Pi^{-1}(\{p\})} M(\xi) \llbracket \xi - p \rrbracket
\end{equation}
is a well defined $Q$-valued vector field with values in the normal bundle, and hence it satisfies property $(ii)$ in Definition \ref{reparametrization:def}. Furthermore, it is evident from the very definition of $M$ that property $(iii)$ in Definition \ref{reparametrization:def} is satisfied as well.

Hence, we are only left with proving that $N$ is Lipschitz continuous and that properties \eqref{reparametrization:th1}-\eqref{reparametrization:th4} are satisfied.

\begin{proposition} \label{vl:prop}
If $c_{0}$ is small enough, depending on $m$, $n$, $r - s$ and $\frac{r}{s}$, then there exists $\tilde{\tau} > 0$ such that the multisection $M$ is $\tilde{\tau}$-vertically limited. Furthermore,
\begin{equation} \label{vl:eq}
\tilde{\tau} \leq C \left( \|N\|_{C^{0}} \|D^{2}\bphi\|_{C^{0}} + \|D\bphi\|_{C^{0}} + \Lip(f) \right),
\end{equation}
where $C = C(m,n)$ and $\|N\|_{C^0} := \sup_{p \in \Sigma} |N(p)| = \sup_{p \in \Sigma} \G(N(p), Q \llbracket 0 \rrbracket)$.
\end{proposition}

\begin{proof}
First, let us exploit again the orthonormal frame $\lbrace \nu_{1}, \dots, \nu_{n} \rbrace$ in order to introduce coordinates on ${\bf U}$. Precisely, we let $\Psi$ denote the map $\xi \in {\bf U} \mapsto \left( \Pi(\xi), v(\xi) \right) \in \Sigma \times \R^n$, where $p := \Pi(\xi)$ is the base point of $\xi$ on $\Sigma$, and $v(\xi) = \left( v^{1}(\xi), \dots, v^{n}(\xi) \right)$ is the set of coordinates of the vector ${\rm v} := \xi - p \in \varkappa_p$ with respect to the basis $\nu_{1}(p), \dots, \nu_{n}(p)$, explicitly given by $v^{i}(\xi) = \langle \xi - p, \nu_{i}(p) \rangle$ for $i = 1,\dots,n$. The map $\Psi$ is a global trivialization of the bundle ${\bf U}$; moreover, since ${\bf \Phi}^{-1}$ is a global chart on $\Sigma$, then, in order to show that $M$ is $\tilde{\tau}$-vertically limited, it suffices to prove that the $Q$-multisection
\[
\widetilde{M} := M \circ \Psi^{-1} \circ \left( {\bf \Phi} \times {\rm id}_{\R^n} \right) \colon B_{s} \times \R^n \to \N
\] 
satisfies the $\tilde{\tau}$-cone condition. In order to see this, fix $\left( x,v \right) \in B_{s} \times \R^n$, and denote by $\xi = \xi(x,v)$ the corresponding point in ${\bf U}$, given by
\begin{equation} \label{vl:1}
\xi(x,v) := {\bf \Phi}(x) + \sum_{i=1}^{n} v^{i} \nu_{i}(x).
\end{equation}
Assume that $\widetilde{M}(x,v) = M(\xi) > 0$: the goal is then to prove that there exists a positive number $\varepsilon$ such that if $(y,w) \in B^{m}_{\varepsilon}(x) \times B^{n}_{\varepsilon}(v)$ satisfies $\widetilde{M}(y,w) > 0$, then necessarily
\begin{equation} \label{vl:2}
|w - v| \leq \tilde{\tau} |y - x|.
\end{equation}

Let $\left( x', v' \right)$ denote the coordinates of $\xi$ in the standard reference frame on $\R^{m+n}$, that is $x' := \p_{\pi_0}(\xi)$ and $v' := \p_{\pi_{0}^{\perp}}(\xi)$. Observe that the condition $M(\xi) > 0$ is equivalent to say that $v' \in \spt(f(x'))$, and in fact $M(\xi) = M_{f}(x',v')$. Now, since the $Q$-valued function $f$ is $\ell$-Lipschitz continuous, $M_f$ satisfies the $\ell$-cone condition, and thus there exists $\delta > 0$ such that if $(y',w') \in B^{m}_{\delta}(x') \times B^{n}_{\delta}(v')$ is such that $M_{f}(y',w') > 0$ then
\begin{equation} \label{vl:3}
|w' - v'| \leq \ell |y' - x'|.
\end{equation}
We first claim the following: there exists $0 < \varepsilon = \varepsilon(\delta,m,n)$ with the property that if 
\[\zeta = {\bf \Phi}(y) + \sum_{i=1}^{n} w^{i} \nu_{i}(y) \hspace{0.5cm} \mbox{ with } (y,w) \in B^{m}_{\varepsilon}(x) \times B^{n}_{\varepsilon}(v),
\]
then
\[
|\p_{\pi_{0}}(\zeta) - x'| < \delta, \hspace{0.5cm} |\p_{\pi_{0}^{\perp}}(\zeta) - v'| < \delta.
\]
This can be immediately seen by estimating:
\begin{equation} \label{vl:4}
\begin{split}
|\zeta - \xi| &\leq |{\bf \Phi}(y) - {\bf \Phi}(x)| + \sum_{i=1}^{n} |w^{i} \nu_{i}(y) - v^{i}\nu_{i}(x)| \\
&\leq \left( 1 + \|D\bphi\|_{C^0} \right) |y - x| + \sum_{i=1}^{n} \left( |w^{i}| |\nu_{i}(y) - \nu_{i}(x)| + |v^{i} - w^{i}| \right) \\
&\leq C \left( 1 + \| D\bphi \|_{C^1} \right) |y - x| + C |w - v|,
\end{split}
\end{equation}
where $C = C(m,n)$ is a geometric constant. The conclusion immediately follows, since $|\p_{\pi_{0}}(\zeta) - x'| = |\p_{\pi_{0}}(\zeta - \xi)| \leq |\zeta - \xi|$ and $|\p_{\pi_{0}^{\perp}}(\zeta) - v'| = |\p_{\pi_{0}^{\perp}}(\zeta - \xi)| \leq |\zeta - \xi|$.

Now, let $(y,w)$ be any point in $B^{m}_{\varepsilon}(x) \times B^{n}_{\varepsilon}(v)$ such that for the corresponding $\zeta \in {\bf U}$ one has $M(\zeta) > 0$. By the above claim, if we set $y' := \p_{\pi_{0}}(\zeta)$ and $w' := \p_{\pi_{0}^{\perp}}(\zeta)$, then $(y',w') \in B^{m}_{\delta}(x') \times B^{n}_{\delta}(v')$, and thus the condition $M_{f}(y',w') > 0$ implies that \eqref{vl:3} holds. Hence, we proceed with the proof of \eqref{vl:2}. For any $i = 1,\dots,n$, one has:
\begin{equation} \label{vl:5}
\begin{split}
|w^{i} - v^{i}| &= \left| \langle \zeta - {\bf \Phi}(y), \nu_{i}(y) \rangle - \langle \xi - {\bf \Phi}(x), \nu_{i}(x) \rangle \right| \\
&\leq |\langle \xi - {\bf \Phi}(x), \nu_{i}(x) - \nu_{i}(y) \rangle| + |\langle \xi - \zeta, \nu_{i}(y) \rangle| + |\langle {\bf \Phi}(x) - {\bf \Phi}(y), \nu_{i}(y) \rangle|. \\
\end{split} 
\end{equation}
Now, since $\xi \in M_{{\bf \Phi}(x)}$, the vector $\xi - {\bf \Phi}(x)$ is in the support of $N({\bf \Phi}(x))$, and thus
\[
|\xi - {\bf \Phi}(x)| \leq |N({\bf \Phi}(x))|.
\]
Therefore, if we apply Lemma \ref{trivialization} we easily estimate
\begin{equation} \label{vl:5a}
|\langle \xi - {\bf \Phi}(x), \nu_{i}(x) - \nu_{i}(y) \rangle| \leq C \| N \|_{C^{0}(\Sigma)} \| D^{2} \bphi \|_{C^{0}} |y - x|.
\end{equation}
In order to estimate the second and third term of \eqref{vl:5}, instead, we first decompose both $\xi - \zeta$ and ${\bf \Phi}(x) - {\bf \Phi}(y)$ by projecting them onto the planes $\pi_0$ and $\pi_{0}^{\perp}$. Then, we use \eqref{wp:6} to conclude that
\begin{equation} \label{vl:5b}
\begin{split}
|\langle \xi - \zeta, \nu_{i}(y) \rangle| &\leq |\langle y' - x', \p_{\pi_{0}}(\nu_{i}(y)) \rangle| + |\langle w' - v', \p_{\pi_{0}^{\perp}}(\nu_{i}(y)) \rangle| \\
&\leq C |D\bphi(y)| |y' - x'| + |w' - v'| \\
&\overset{\eqref{vl:3}}{\leq} \left( C \| D\bphi \|_{C^0} + \ell \right) |y' - x'|,  
\end{split}
\end{equation}
and analogously
\begin{equation} \label{vl:5c}
\begin{split}
|\langle {\bf \Phi}(x) - {\bf \Phi}(y), \nu_{i}(y) \rangle| &\leq C |D\bphi(y)| |y - x| + |\bphi(y) - \bphi(x)| \\
&\leq C \| D\bphi \|_{C^0} |y - x|. 
\end{split}
\end{equation}

Inserting \eqref{vl:5a}, \eqref{vl:5b} and \eqref{vl:5c} into \eqref{vl:5}, we then conclude the following estimate:
\begin{equation} \label{vl:6}
|w^{i} - v^{i}| \leq C \left( \|N\|_{C^0} \|D^{2}\bphi\|_{C^0} + \|D\bphi\|_{C^0} \right) |y - x| + \left( C \|D\bphi\|_{C^0} + \ell \right) |y' - x'|.
\end{equation}
Therefore, in order to conclude, we need to bound:
\begin{equation} \label{vl:7}
\begin{split}
|y' - x'| &= |\p_{\pi_0}(\zeta) - \p_{\pi_0}(\xi)| \\
&= \left|y + \sum_{i=1}^{n} w^{i} \p_{\pi_0}(\nu_{i}(y)) - x - \sum_{i=1}^{n} v^{i} \p_{\pi_0}(\nu_{i}(x))\right| \\
&\leq |y - x| + \sum_{i=1}^{n} \left( |w^{i}| |\nu_{i}(y) - \nu_{i}(x)| + |w^{i} - v^{i}| |\p_{\pi_0}(\nu_{i}(x))| \right) \\
&\leq \left( 1 + C \|D^{2}\bphi\|_{C^0} \right) |y - x| + C \|D\bphi\|_{C^0} |w - v|.
\end{split}
\end{equation}
If we combine \eqref{vl:6} and \eqref{vl:7}, after standard algebraic computations we obtain:
\begin{equation} \label{vl:8}
|w - v| \leq C \left( \|N\|_{C^0} \|D^{2}\bphi\|_{C^0} + \| D\bphi \|_{C^0} + \Lip(f) \right) |y - x| + C c_{0}^{2} |w - v|,
\end{equation}
where the constant $C$ appearing on the right-hand side of the inequality is purely geometric, and, in particular, does not depend on $c_0$. This allows us to conclude that if $c_0$ is such that
\begin{equation} \label{vl:9}
C c_{0}^{2} \leq \frac{1}{2},
\end{equation}
then a cone condition for $\widetilde{M}$ holds in the form
\begin{equation} \label{vl:10}
|w - v| \leq \tilde{\tau} |y - x|
\end{equation}
with $\tilde{\tau}$ as in \eqref{vl:eq} for any $(y,w)$ in a suitable neighborhood of $(x,v)$ such that $\widetilde{M}(y,w) > 0$. Since the choice of the point $(x,v)$ was arbitrary, the proof is complete.
\end{proof}

\begin{proof}[Proof of Theorem \ref{reparametrization:thm}]
We start proving that $N$ is Lipschitz continuous. Let $c_{0}$ be such that Proposition \ref{well-posedness} and Proposition \ref{vl:prop} both hold. We make the following

\hspace{0.5cm} \textbf{Claim.} For every $p \in \Sigma$ there exists an open neighborhood $\U_{p}$ of $p$ in $\Sigma$ such that
\begin{equation} \label{proof:claim}
\G(N(q), N(p)) \leq \sqrt{Q} \tilde{\tau}' d_{\Sigma}(q,p) \hspace{0.5cm} \mbox{for every } q \in \U_{p},
\end{equation}
where $\tilde{\tau}'$ satisfies the same estimate as in equation \eqref{vl:eq} and $d_{\Sigma}(\cdot, \cdot)$ is the geodesic distance function on $\Sigma$. In order to see this, fix a point $p \in \Sigma$ and let $M_{p}$ denote, as usual, the set of points $\xi \in {\bf U}$ such that $\Pi(\xi) = p$ and $M(\xi) > 0$. Assume that $M_{p} = \lbrace \xi_{1}, \dots \xi_{J} \rbrace$. If $p = {\bf \Phi}(x)$, then for any $j = 1,\dots,J$ one has
\begin{equation} \label{proof:1}
\xi_{j} = {\bf \Phi}(x) + \sum_{i=1}^{n} v_{j}^{i} \nu_{i}(x).
\end{equation}
By Proposition \ref{vl:prop}, there exist neighborhoods $U_{j}$ of $x$ in $B_{s}$ and $V_{j}$ of $v_{j} := \left( v_{j}^{1}, \dots, v_{j}^{n} \right)$ in $\R^{n}$ such that if 
\begin{equation} \label{proof:2}
\zeta = \zeta(y,w) := {\bf \Phi}(y) + \sum_{i=1}^{n} w^{i} \nu_{i}(y) \hspace{0.5cm} \mbox{with } (y,w) \in U_{j} \times V_{j}
\end{equation}
is such that $M(\zeta) > 0$ then necessarily
\begin{equation} \label{proof:3}
|w - v_{j}| \leq \tilde{\tau} |y - x|.
\end{equation}
Let $\left( x(\zeta), v(\zeta) \right)$ denote the inverse mapping of $\zeta(x,v)$, given by
\begin{equation} \label{proof:4}
x(\zeta) := \p_{\pi_{0}} \circ \Pi(\zeta), \hspace{0.5cm} v^{i}(\zeta) := \langle \zeta - \Pi(\zeta), \nu_{i}(\Pi(\zeta)) \rangle,
\end{equation}
and let
\begin{equation} \label{proof:5}
\mathcal{V}_{j} := \lbrace \zeta \in {\bf U} \, \colon \, \left( x(\zeta), v(\zeta) \right) \in U_{j} \times V_{j} \rbrace.
\end{equation}
Each $\mathcal{V}_{j}$ is an open neighborhood of $\xi_{j}$, and moreover the cone condition \eqref{proof:3} forces $\mathcal{V}_{j} \cap M_{p} = \{ \xi_{j} \}$. We can also assume without loss of generality that the $\mathcal{V}_{j}$'s are pairwise disjoint. By Proposition \ref{well-posedness}, since $M$ is coherent there exists a neighborhood $\U_{p}$ of $p$ in $\Sigma$ such that
\begin{equation} \label{proof:6}
\sum_{\zeta \in \Pi^{-1}(\{q\}) \cap \mathcal{V}_{j}} M(\zeta) = M(\xi_{j}) \hspace{0.5cm} \mbox{for every } q \in \U_{p}.
\end{equation}
Since 
\begin{equation} \label{proof:7}
\sum_{j=1}^{J} M(\xi_{j}) = Q,
\end{equation}
it is evident that when $q$ is chosen in $\U_{p}$ then any $\zeta \in \Pi^{-1}(\{q\})$ with $M(\zeta) > 0$ must be an element of one and only one $\mathcal{V}_{j}$, and thus we can write
\begin{equation} \label{proof:8}
N(q) = \sum_{j=1}^{J} \sum_{\zeta \in \Pi^{-1}(\{q\}) \cap \mathcal{V}_{j}} M(\zeta) \llbracket \zeta - q \rrbracket,
\end{equation}
whereas
\begin{equation} \label{proof:9}
N(p) = \sum_{j=1}^{J} M(\xi_{j}) \llbracket \xi_{j} - p \rrbracket.
\end{equation}
We can now estimate, for any $\zeta \in \Pi^{-1}(\{q\}) \cap \mathcal{V}_{j}$ with $M(\zeta) > 0$ and $q = {\bf \Phi}(y) \in \U_{p}$:
\begin{equation} \label{proof:10}
\begin{split}
\left| (\zeta - q) - (\xi_{j} - p) \right| &= \left| \sum_{i=1}^{n} w^{i} \nu_{i}(y) - \sum_{i=1}^{n} v_{j}^{i} \nu_{i}(x) \right| \\
&\leq \sum_{i=1}^{n} \left( |w^{i}| |\nu_{i}(y) - \nu_{i}(x)| + |w^{i} - v_{j}^{i}| \right) \\
&\leq C \| N \|_{C^{0}} \| D^{2}\bphi \|_{C^{0}} |y - x| + C |w - v_{j}| \\
&\overset{\eqref{proof:3}}{\leq} C \left( \| N \|_{C^{0}} \| D^{2}\bphi \|_{C^{0}} + \tilde{\tau} \right) |y - x| = \tilde{\tau}' |y - x|. 
\end{split}
\end{equation}
Observe that the constant $C$ appearing in \eqref{proof:10} is purely geometric, and that $\tilde{\tau}'$ also satisfies the bound in \eqref{vl:eq}. It is now evident that
\begin{equation} \label{proof:11}
\G(N(q), N(p))^{2} \leq Q \tilde{\tau}'^{2} |y - x|^{2},
\end{equation}
from which the claim follows because $|y - x| \leq |q - p| \leq d_{\Sigma}(q,p)$.

Now, we show how from the claim one can easily conclude the Lipschitz continuity of $N$ with the required estimates. Fix two distinct points $p,q \in \Sigma$, and let $\gamma \colon \left[ a, b \right] \to \Sigma$ be any (piecewise) smooth curve such that $\gamma(a) = p$ and $\gamma(b) = q$. By the claim, for every $t \in \left[ a, b \right]$ there exists a neighborhood $\U_{\gamma(t)}$ such that
\begin{equation} \label{proof:12}
\G(N(z), N(\gamma(t))) \leq \sqrt{Q} \tilde{\tau}' d_{\Sigma}(z, \gamma(t)) \hspace{0.5cm} \mbox{for every } z \in \U_{\gamma(t)}.
\end{equation}
Since $\gamma$ is continuous, there exist numbers $\delta_{t}$ such that
\begin{equation} \label{proof:13}
I_{t} := \left( t - \delta_{t}, t + \delta_{t} \right) \subset \gamma^{-1}(\U_{\gamma(t)}).
\end{equation}
The family $\lbrace I_{t} \rbrace$ is an open covering of the interval $\left[ a,b \right]$, and thus by compactness we can extract a finite subcovering $\lbrace I_{t_i} \rbrace_{i=0}^{K}$. We may assume, refining the subcovering if necessary, that an interval $I_{t_i}$ is not completely contained in an interval $I_{t_j}$ if $i \neq j$. If we relabel the indices of the points $t_{i}$ in a non-decreasing order, and thus in such a way that $\gamma(t_{i})$ preceeds $\gamma(t_{i+1})$, we can now choose an auxiliary point $s_{i,i+1}$ in $I_{t_i} \cap I_{t_{i+1}} \cap \left( t_{i}, t_{i+1} \right)$ for each $i = 0,\dots,K-1$. We can finally conclude:
\begin{equation} \label{proof:14}
\begin{split}
\G(N(q), &N(p)) \leq \G(N(p), N(\gamma(t_{0}))) \\
&+ \sum_{i=0}^{K-1} \left( \G(N(\gamma(t_{i})), N(\gamma(s_{i,i+1}))) + \G(N(\gamma(s_{i,i+1})), N(\gamma(t_{i+1}))) \right) + \G(N(\gamma(t_K)), N(q)) \\
& \quad \quad \, \overset{\eqref{proof:13}}{\leq} \sqrt{Q} \tilde{\tau}' \mathscr{L}(\gamma),
\end{split}
\end{equation} 
where $\mathscr{L}(\gamma)$ is the length of the curve $\gamma$. Minimizing among all the piecewise smooth curves $\gamma$ joining $p$ to $q$, one finally obtains
\begin{equation} \label{proof:15}
\G(N(q), N(p)) \leq \sqrt{Q} \tilde{\tau}' d_{\Sigma}(q,p), 
\end{equation}
that is
\begin{equation} \label{proof:16}
\Lip(N) \leq \sqrt{Q} \tilde{\tau}'.
\end{equation}
The estimate \eqref{reparametrization:th1} is now just a consequence of \eqref{vl:eq}.

In order to complete the proof, we are left with showing the validity of \eqref{reparametrization:th2}, \eqref{reparametrization:th3} and \eqref{reparametrization:th4}. This can be done by reproducing verbatim the proof suggested by De Lellis and Spadaro in \cite{DLS13a}; the arguments will be presented here only for completeness.

We start with the proof of \eqref{reparametrization:th2} and \eqref{reparametrization:th3}. Fix a point $x \in B_{s}$, and let $p := {\bf \Phi}(x) \in \Sigma$. Observe that, by \eqref{reparametrized_multi} and \eqref{rep:Ndef}, the definition of the value of $N(p)$ does not change if we replace $\bphi$ with its first order Taylor expansion at $x$, since this operation preserves the fiber $\Pi^{-1}(\{p\})$. Furthermore, we can assume without loss of generality that $x = 0$ and $\bphi(0) = 0$. We will still use the symbols $\pi_{0}$ and $\pi_{0}^{\perp}$ to denote the planes $\R^{m} \times \{0\} \simeq \R^{m}$ and $\{0\} \times \R^{n} \simeq \R^{n}$ respectively, whereas the tangent space $T_{0}\Sigma$ and its orthogonal complement $T_{0}^{\perp}\Sigma$ will be denoted $\pi$ and $\varkappa$. 
Now, concerning the estimate \eqref{reparametrization:th2}, assume that $f(0) = \sum_{l=1}^{Q} \llbracket v_{l} \rrbracket$, set $\xi_{l} := \left( 0, v_{l} \right) \in \pi_{0} \times \pi_{0}^{\perp}$ and $q_{l} := \p_{\pi}(\xi_{l})$. If $N(q_{l}) = \sum_{j=1}^{Q} \llbracket \zeta_{l,j} \rrbracket$, then there is an index $j(l)$ such that $\zeta_{l,j(l)} = \xi_{l}$. If the point $\zeta_{l,j(l)}$ has coordinates $\left( q_{l}, v_{l}' \right)$ in the frame $\pi \times \varkappa$, we get
\begin{equation}
\begin{split}
|v_{l}| &\leq |q_{l}| + |v_{l}'| \leq |q_{l}| + |N(0)| + \G(N(0), N(q_{l})) \\
&\leq |N(0)| + \left( 1 + \Lip(N) \right) |q_{l}| \leq |N(0)| + C \left( 1 + \Lip(N) \right) \| D \bphi \|_{C^0} |v_{l}|,
\end{split}
\end{equation} 
where we have used that $q_{l} = |\p_{\pi}(\xi_{l})| \leq C |D \bphi(0)| |\xi_{l}| = C \| D \bphi \|_{C^{0}} |v_{l}|$. Now, we use \eqref{reparametrization:th1} with $\bphi$ linear to estimate
\begin{equation}
\Lip(N) \leq C \left( \| D \bphi \|_{C^{0}} + \Lip(f) \right) \leq C c_{0}.
\end{equation}
Thus, we conclude
\begin{equation}
|v_{l}| \leq |N(0)| + C (1 + C c_{0}) c_{0} |v_{l}|.
\end{equation}
Since the constant $C$ is purely geometric and does not depend on $c_{0}$, we deduce that if $c_{0}$ is sufficiently small then $|v_{l}| \leq 2 |N(0)|$. Summing over $l \in \{ 1, \dots Q \}$ we obtain $|f(0)| \leq 2 \sqrt{Q} |N(0)|$. The proof of the other inequality, namely $|N(0)| \leq 2 \sqrt{Q} |f(0)|$, is analogous, reversing the roles of the systems of coordinates $\pi_{0} \times \pi_{0}^{\perp}$ and $\pi \times \varkappa$. This concludes the proof of \eqref{reparametrization:th2}.

We proceed with the proof of \eqref{reparametrization:th3}. Assume once again that $f(0) = \sum_{l=1}^{Q} \llbracket v_{l} \rrbracket$, and write $N(0) = \sum_{l=1}^{Q} \llbracket \xi_{l} \rrbracket$. For every $l \in \{1, \dots, Q\}$, we set $x_{l} := \p_{\pi_0}(\xi_l)$, $w_{l} := \p_{\pi_{0}^{\perp}}(\xi_l)$ and $w_{l}' := \p_{\varkappa}(\xi_{l})$, so that the point $\xi_{l}$ is represented by coordinates $\left( x_{l}, w_{l} \right)$ in the standard reference frame $\pi_{0} \times \pi_{0}^{\perp}$ and by coordinates $\left( 0, w_{l}' \right)$ in the frame $\pi \times \varkappa$. As usual, we have:
\begin{equation} \label{lungh_orizz} 
|x_{l}| = |\p_{\pi_0}(\xi_l)| \leq C |D \bphi(0)| |\xi_{l}| \leq C |D \bphi(0)| |N(0)| =: \rho.
\end{equation}

Using these notations, one has $|\bfeta \circ N(0)| = Q^{-1} \left| \sum_{l} w_{l}' \right|$. On the other hand, under our usual smallness assumptions on the size of $c_{0}$, we can also assume that the operator norm of the linear and invertible transformation $L \colon \pi_{0}^{\perp} \to \varkappa$ is bounded by $2$. Thus, we can further estimate $|\bfeta \circ N(0)| \leq 2 Q^{-1} \left| \sum_{l=1}^{Q} w_{l} \right|$, so that in order to get \eqref{reparametrization:th3} it would suffice to prove the following:
\begin{equation} \label{stima_media}
\left| \sum_{l=1}^{Q} w_{l} \right| \leq \left| \sum_{l=1}^{Q} v_{l} \right| + C \Lip(f) \rho.
\end{equation}
In order to show the validity of \eqref{stima_media}, we notice that if we set $h := \Lip(f) \rho$, then we can decompose $f(0) = \sum_{j=1}^{J} \llbracket T_{j} \rrbracket$, where each $T_{j} \in \A_{Q_{j}}(\R^{n})$, $\sum_{j=1}^{J} Q_{j} = Q$ and with the property that:
\begin{itemize}
\item[$(i)$] ${\rm diam}(T_{j}) \leq 4Qh$;

\item[$(ii)$] $|y - z| > 4h$ for all $y \in \spt(T_{i})$ and $z \in \spt(T_{j})$ when $i \neq j$.
\end{itemize}

This claim can be justified with the following simple argument. First, we order the vectors $v_{l}$, and then we partition them in subcollections $T_{j}$ according to the following algorithm: $T_{1}$ contains $v_{1}$ and any other vector $v_{\ell}$ for which there exists a chain $v_{l(1)}, \dots v_{l(k)}$ with $l(1) = 1$, $l(k) = \ell$ and $|v_{l(i+1)} - v_{l(i)}| \leq 4h$ for every $i = 0, \dots, k-1$. By construction, ${\rm diam}(T_{1}) \leq 4Qh$, and if $\spt(T_{1}) = \spt(f(0))$ then we are finished. Otherwise, we construct $T_{2}$ applying the same algorithm to the vectors in $\spt(f(0)) \setminus \spt(T_{1})$. The construction of the algorithm guarantees that also property $(ii)$ is satisfied.

Given the above decomposition of $f(0)$, we observe that from the choice of the constants it follows that in the ball $B_{\rho}$ the function $f$ decomposes into the sum $f = \sum_{j=1}^{J} \llbracket f^{j} \rrbracket$ of $J$ Lipschitz functions $f^{j} \colon B_{\rho} \to \A_{Q_{j}}(\R^{n})$ with $\Lip(f^{j}) \leq \Lip(f)$ for every $j$. In agreement with this decomposition, also the graph ${\rm Gr}(f|_{B_\rho})$ separates into the union $\bigcup_{j=1}^{J} {\rm Gr}(f^{j})$. By the definition of the vector field $N$ (cf. again \eqref{reparametrized_multi} and \eqref{rep:Ndef}), the support of $N(0)$ contains points from each of these sets; furthermore, if $\xi \in \spt(N(0)) \cap {\rm Gr}(f^{j})$ then $M(\xi) = M_{f^{j}}(\p_{\pi_0}(\xi), \p_{\pi_{0}^{\perp}}(\xi))$. It follows that also $N(0)$ can be decomposed into $N(0) = \sum_{j=1}^{J} \sum_{i=1}^{Q_{j}} \llbracket \xi^{j}_{i} \rrbracket$ with the property that $\xi^{j}_{i} \in {\rm Gr}(f^{j})$ for every $i = 1,\dots,Q_{j}$. 

Now, by the definition of the distance $\G$, for each $\xi^{j}_{i} \in \spt(N(0))$ there exists a point $v_{k(j,i)} \in \spt(f^{j}(0))$ such that $|w^{j}_{i} - v_{k(j,i)}| \leq \G(f^{j}(x^{j}_{i}), f^{j}(0)) \leq \Lip(f) |x^{j}_{i}| \overset{\eqref{lungh_orizz}}{\leq} \Lip(f) \rho = h$. Hence, we conclude:
\[
\begin{split}
\left| \sum_{l=1}^{Q} w_{l} \right| &= \left| \sum_{j=1}^{J} \sum_{i=1}^{Q_j} w^{j}_{i} \right| \leq \left| \sum_{j=1}^{J} \sum_{i=1}^{Q_j} v^{j}_{i} \right| + \sum_{j=1}^{J} \sum_{i=1}^{Q_j} |w^{j}_{i} - v^{j}_{i}| \\
&\leq \left| \sum_{l=1}^{Q} v_{l} \right| + \sum_{j=1}^{J} \sum_{i=1}^{Q_j} \left( |w^{j}_{i} - v_{k(j,i)}| + |v_{k(j,i)} - v^{j}_{i}| \right) \leq \left| \sum_{l=1}^{Q} v_{l} \right| + Ch,
\end{split}
\] 
where we used that ${\rm diam}(f^{j}(0)) \leq 4Qh$. This proves \eqref{stima_media} and concludes the proof of \eqref{reparametrization:th3}.

Finally, we show that \eqref{reparametrization:th4} holds. Let $x \in B_{s}$, and assume that $\left( x, \bfeta \circ f(x) \right) = p + {\rm v}$ for some $p \in \Sigma$ and ${\rm v} \in T_{p}^{\perp}\Sigma$. Now, if ${\rm v} = 0$ then the above assumption implies that $\bfeta \circ f(x) = \bphi(x)$, and thus \eqref{reparametrization:th4} reduces to the first inequality in \eqref{reparametrization:th2}. On the other hand, if ${\rm v} \neq 0$ then we shift $\Sigma$ to $\tilde{\Sigma} := {\rm v} + \Sigma$. Then, if we apply Theorem \ref{reparametrization:thm} with $\tilde{\Sigma}$ in place of $\Sigma$ we obtain a vector field $\tilde{N}$ which satisfies $\tilde{N}(p + {\rm v}) = \sum_{l} \llbracket N_{l}(p) - {\rm v} \rrbracket$. Hence, $\G(N(p), Q \llbracket {\rm v} \rrbracket) = \G(\tilde{N}(p+{\rm v}), Q \llbracket 0 \rrbracket)$, which reduces the problem again to the case ${\rm v} = 0$. This completes the proof of Theorem \ref{reparametrization:thm}. 
\end{proof}

\appendix

\section{Commutation of push-forward and boundary in the $Q$-valued setting: an alternative proof} \label{sec:commutation}

In this Appendix we are going to provide an alternative proof (with respect to the existing literature, see \cite{Almgren00} and \cite{DLS13a}) of the fact that the multi-valued push-forward operator on Lipschitz submanifolds commutes with the boundary operator in the sense of currents. As anticipated, the proof is based on a double induction, both on the dimension $m$ of the manifold and on the number $Q$ of values that the function takes.

\begin{proof}[Proof of Theorem \ref{pf_bdry:thm}]
First observe that since every Lipschitz manifold can be triangulated, and since the statement is invariant under bi-Lipschitz homeomorphisms, it is enough to prove the theorem with $\Sigma = \left[ 0,1 \right]^{m}$. Furthermore, it suffices to show that the theorem holds in the case of the currents associated to graphs. Indeed, suppose to know that $\partial {\bf G}_{u} = {\bf G}_{u|_{\partial \Sigma}}$, and let ${\bf p} \colon \R^{d} \times \R^{n} \to \R^{n}$ be the orthogonal projection onto the second components. Then, it is immediate to see that
\[
{\bf p}_{\sharp} {\bf G}_{u} = {\bf p}_{\sharp} {\bf T}_{{\rm Id} \times u} = {\bf T}_{{\bf p} \circ ({\rm Id} \times u)} = {\bf T}_{u},
\]
where, for given Lipschitz $F \colon \R^{d} \to \A_{Q}(\R^{n})$ and $\phi \colon \R^{n} \to \R^{k}$, we have used the notation $\phi \circ F$ for the $Q$-valued function $\phi \circ F(p) := \sum_{l=1}^{Q} \llbracket \phi(F_{l}(p)) \rrbracket \in \A_{Q}(\R^{k})$. Then, using that push-forward and boundary do commute in the case of single valued Lipschitz functions, one readily concludes
\[
\partial {\bf T}_{u} = \partial {\bf p}_{\sharp} {\bf G}_{u} = {\bf p}_{\sharp} \partial {\bf G}_{u} = {\bf p}_{\sharp} {\bf G}_{u|_{\partial \Sigma}} = {\bf T}_{u|_{\partial \Sigma}}.
\]

Hence, we show that $\partial {\bf G}_{u} = {\bf G}_{u|_{\partial \Sigma}}$. The proof is by induction on both $m$ and $Q$. If $Q = 1$, the result is classical. On the other hand, the case $m = 1$ is a consequence of \cite[Proposition 1.2]{DLS11a}: if $u \colon \left[ 0,1 \right] \to \A_{Q}(\R^{n})$ is Lipschitz, then there exist Lipschitz functions $u_{1},\dots,u_{Q} \colon \left[ 0,1 \right] \to \R^{n}$ such that $u = \sum_{l=1}^{Q} \llbracket u_{l} \rrbracket$. Therefore, ${\bf T}_{u} = \sum_{l} (u_{l})_{\sharp} \llbracket (0,1) \rrbracket$, and thus
\[
\partial {\bf T}_{u} = \sum_{l} \partial (u_{l})_{\sharp} \llbracket (0,1) \rrbracket = \sum_{l} (u_{l})_{\sharp} \left( \llbracket 1 \rrbracket - \llbracket 0 \rrbracket \right) = \sum_{l} \left( \llbracket u_{l}(1) \rrbracket - \llbracket u_{l}(0) \rrbracket \right) = {\bf T}_{u|_{\partial \Sigma}}.  
\]

Then, we make the following inductive hypotheses:
\begin{itemize}
\item[$(H1)$] the theorem is true when $\dim(\Sigma) \leq m-1$,

\item[$(H2)$] the theorem is true for $\dim(\Sigma) = m$ when the function $u$ takes $Q^{*}$ values for every $Q^{*} < Q$,
\end{itemize}
and we show that the theorem is true for $\left( m,Q \right)$. In order to do this, we consider a dyadic decomposition of $\Sigma = \left[ 0,1 \right]^{m}$ in $m$-cubes of side length $2^{-h}$ with $h \in \N$, and for any integer vector $k \in \{ 0,1,\dots,2^{h}-1 \}^{m}$ we let $C_{h,k}$ be the cube $C_{h,k} := 2^{-h} \left( k + \left[ 0,1 \right]^{m} \right)$.

Now, for fixed $h$, let $\mathscr{B}_{h}$ be the set of all $k \in \{ 0,1,\dots,2^{h}-1 \}^{m}$ such that on the corresponding cube $C_{h,k}$ one has
\begin{equation} \label{pf_bdry:1}
\max_{p \in C_{h,k}} {\rm diam}(u(p)) > 3(Q-1)\Lip(u)2^{-h}\sqrt{m}.
\end{equation}  
By \cite[Proposition 1.6]{DLS11a}, if $k \in \mathscr{B}_{h}$ then on the cube $C_{h,k}$ the function $u$ is well separated into the sum
\begin{equation}
u|_{C_{h,k}} = \llbracket u_{k,Q_1} \rrbracket + \llbracket u_{k,Q_2} \rrbracket,
\end{equation}
where $u_{k,Q_1} \in \Lip(C_{h,k}, \A_{Q_1}(\R^n))$, $u_{k, Q_2} \in \Lip(C_{h,k}, \A_{Q_2}(\R^{n}))$ and $Q_{1}, Q_{2} < Q$. Therefore, by the inductive hypothesis $(H2)$ we can conclude that
\begin{equation} \label{pf_bdry:2}
\partial {\bf G}_{u|_{C_{h,k}}} = {\bf G}_{u|_{\partial C_{h,k}}}
\end{equation}
for every $k \in \mathscr{B}_{h}$.

If on the other hand $k \notin \mathscr{B}_{h}$, consider the affine homotopy $\sigma \colon \left[ 0,1 \right] \times C_{h,k} \times \R^{n} \to \R^{d} \times \R^{n}$ defined by
\begin{equation}
\sigma(t,p,v) := \left( p, (1-t)\bfeta \circ u(p) + t v \right), 
\end{equation}
and define the current
\begin{equation}
R_{k} := Q {\bf G}_{(\bfeta \circ u)|_{C_{h,k}}} + \sigma_{\sharp}(\llbracket (0,1) \rrbracket \times {\bf G}_{u|_{\partial C_{h,k}}}).
\end{equation}
Here, $\bfeta \circ u$ denotes the (single valued) Lipschitz function $\bfeta \circ u \colon \Sigma \to \R^{n}$ given by
\[
\bfeta \circ u(p) := \bfeta(u(p)) = \frac{1}{Q} \sum_{l=1}^{Q} u_{l}(p).
\]
Since $\bfeta \circ u$ is a classical Lipschitz function, the classical commutation rule of push-forward and boundary gives 
\begin{equation} \label{pf_bdry:3}
\partial (Q {\bf G}_{(\bfeta \circ u)|_{C_{h,k}}}) = Q {\bf G}_{(\bfeta \circ u)|_{\partial C_{h,k}}}.
\end{equation}
On the other hand, the homotopy formula \cite[(26.22)]{Sim83} yields
\begin{equation} \label{pf_bdry:4}
\partial \sigma_{\sharp}(\llbracket (0,1) \rrbracket \times {\bf G}_{u|_{\partial C_{h,k}}}) = {\bf G}_{u|_{\partial C_{h,k}}} - Q {\bf G}_{(\bfeta \circ u)|_{\partial C_{h,k}}} - \sigma_{\sharp}(\llbracket (0,1) \rrbracket \times \partial {\bf G}_{u|_{\partial C_{h,k}}}).
\end{equation}
Since $\partial C_{h,k}$ is the union of $(m-1)$-dimensional cubes, the inductive hypothesis $(H1)$ ensures that in fact $\partial {\bf G}_{u|_{\partial C_{h,k}}} = 0$, and thus the last addendum in the r.h.s. of equation \eqref{pf_bdry:4} vanishes. Combining \eqref{pf_bdry:3} and \eqref{pf_bdry:4} therefore yields
\begin{equation} \label{pf_bdry:5}
\partial R_{k} = {\bf G}_{u|_{\partial C_{h,k}}}. 
\end{equation}

For every $h \in \N$, define the current
\begin{equation} \label{pf_bdry:6}
T_{h} := \sum_{k \in \mathscr{B}_{h}} {\bf G}_{u|_{C_{h,k}}} + \sum_{k \notin \mathscr{B}_{h}} R_{k}, 
\end{equation}
and notice that by \eqref{pf_bdry:2} and \eqref{pf_bdry:5} one has
\begin{equation} \label{pf_bdry:7}
\partial T_{h} = \sum_{k} {\bf G}_{u|_{\partial C_{h,k}}} = {\bf G}_{u|_{\partial \Sigma}}
\end{equation}
because the common faces to adjacent cubes have opposite orientations. Furthermore, it is easy to see that for every $h \in \N$ and for every $k \in \mathscr{B}_{h}$ one has
\begin{equation} \label{pf_bdry:8}
\M({\bf G}_{u|_{C_{h,k}}}) \leq C (1 + \Lip(u))^{m} \Ha^{m}(C_{h,k}) \leq C (2^{-h})^{m}, 
\end{equation}
whereas, using \cite[(26.23)]{Sim83} to estimate $\M(\sigma_{\sharp}(\llbracket (0,1) \rrbracket \times {\bf G}_{u|_{\partial C_{h,k}}}))$, we get
\begin{equation} \label{pf_bdry:9}
\begin{split}
\M(R_{k}) &\leq C (2^{-h})^{m} + C \M({\bf G}_{u|_{\partial C_{h,k}}}) \sup_{( p,v ) \in {\rm Gr}(u|_{\partial C_{h,k}})} |(p,v) - (p, \bfeta \circ u(p))| \\
&\leq C (2^{-h})^{m} + C (2^{-h})^{m-1} \sup_{p \in \partial C_{h,k}} \max_{l \in \{1,\dots,Q\}} |u_{l}(p) - \bfeta \circ u(p)| \\
&\leq C (2^{-h})^{m} + C (2^{-h})^{m-1} \sup_{p \in \partial C_{h,k}} {\rm diam}(u(p)) \\
&\leq C (2^{-h})^{m}
\end{split} 
\end{equation}
if $k \notin \mathscr{B}_{h}$, for a constant $C = C(m,Q, \Lip(u))$.

By equations \eqref{pf_bdry:6}, \eqref{pf_bdry:7}, \eqref{pf_bdry:8} and \eqref{pf_bdry:9} we immediately conclude that
\begin{equation} \label{pf_bdry:10}
\M(T_{h}) + \M(\partial T_{h}) \leq C,
\end{equation}
where $C = C(m,Q,\Lip(u))$ is a constant independent of $h$. It then follows from the compactness theorem for integral currents \cite[Theorem 27.3]{Sim83} that when $h \uparrow \infty$ a subsequence of the $T_{h}$'s converges to an integral current $T$ such that $\partial T = {\bf G}_{u|_{\partial \Sigma}}$.

We are only left to prove that in fact $T = {\bf G}_{u}$. Since clearly $\spt(T) \subset {\rm Gr}(u)$ and $T$ is integral, we have that $T = \llbracket {\rm Gr}(u), \vec{\eta}, \Theta_{T} \rrbracket$ and ${\bf G}_{u} = \llbracket {\rm Gr}(u), \vec{\eta}, \Theta_{{\bf G}_{u}} \rrbracket$. We only need to show that $\Theta_{T}(p,v) = \Theta_{{\bf G}_{u}}(p,v)$ at $\Ha^{m}$-a.e. $(p,v) \in {\rm Gr}(u)$. Let $p \in \Sigma$, and denote by ${\rm D}_{Q}(u)$ the closed set
\[
{\rm D}_{Q}(u) := \lbrace p \in \Sigma \, \colon \, u(p) = Q \llbracket v \rrbracket \mbox{ for some } v \in \R^{n} \rbrace 
\]
of multiplicity $Q$ points of the function $u$. If $p \notin {\rm D}_{Q}(u)$, then there exists a suitably large $\bar{h}$ such that for every $h \geq \bar{h}$ one has $p \in C_{h,k}$ for some $k \in \mathscr{B}_{h}$, and thus it follows naturally that $\Theta_{T}(p, u_{l}(p)) = \Theta_{{\bf G}_{u}}(p, u_{l}(p))$ for every $l$. Hence, if $\Ha^{m}({\rm D}_{Q}(u)) = 0$ then we are done. Otherwise, consider the $1$-Lipschitz orthogonal projection on the first components $\bar{\p} \colon \R^{d} \times \R^{n} \to \R^{d}$. One has that $\bar{\p}_{\sharp}T = \bar{\Theta}_{T} \llbracket \Sigma \rrbracket$ and $\bar{\p}_{\sharp}{\bf G}_{u} = \bar{\Theta}_{{\bf G}_{u}} \llbracket \Sigma \rrbracket$, with
\[
\bar{\Theta}_{T}(x) = \sum_{(x,v) \in {\rm Gr}(u)} \Theta_{T}(x,v) \hspace{0.2cm} \mbox{ and } \hspace{0.2cm} \bar{\Theta}_{{\bf G}_{u}}(x) = \sum_{(x,v) \in {\rm Gr}(u)} \Theta_{{\bf G}_{u}}(x,v) \hspace{0.2cm} { for } \, \Ha^{m}\mbox{-a.e. } x \in \Sigma.
\]
In particular, for $\Ha^{m}$-a.e. $p \in {\rm D}_{Q}(u)$, if $u(p) = Q \llbracket v(p) \rrbracket$ then $\bar{\Theta}_{T}(p) = \Theta_{T}(p, v(p))$ and $\bar{\Theta}_{{\bf G}_{u}}(p) = \Theta_{{\bf G}_{u}}(p,v(p))$. On the other hand, by the definitions of $u$ and $T_{h}$ it also holds $\bar{\p}_{\sharp} {\bf G}_{u} = Q \llbracket \Sigma \rrbracket = \bar{\p}_{\sharp} T_{h}$ for every $h$. Since $T$ is the limit of (a subsequence of) the $T_{h}$, then necessarily $\bar{\Theta}_{{\bf G}_{u}}(p) = Q = \bar{\Theta}_{T}(p)$ $\Ha^{m}$-a.e. on $\Sigma$, and thus finally $\Theta_{{\bf G}_{u}}(p, v(p)) = Q = \Theta_{T}(p, v(p))$ for $\Ha^{m}$-a.e. $p \in {\rm D}_{Q}(u)$. This completes the proof.
\end{proof}

\section{Multi-valued push-forward of flat chains}\label{ssec:pf_fc} 

The goal of this Appendix is to extend the definition of multiple valued push-forward to the class of integral flat chains. As already mentioned, the existence of a multi-valued push-forward operator acting on flat chains has already been investigated by Almgren in \cite[Section 1.6]{Almgren00}. In what follows, we deduce it as a rather immediate consequence of Theorem \ref{pf_bdry:thm} and of the polyhedral approximation of flat chains \cite[Theorem 4.2.22]{Federer69}.

The first observation is that using similar arguments to those carried in Remark \ref{ext:lip_man}, it is not difficult to extend the results of \S\, \ref{sec:pf} to multi-valued push-forwards of general integer rectifiable currents. This was already observed by De Lellis and Spadaro in \cite{DLS13a}, without going further into the details. Indeed, assume that $\Omega \subset \R^{d}$ is open and if $T \in \Rc_{m}(\Omega)$. Then, there exist a sequence of $C^{1}$ oriented $m$-dimensional submanifolds $\Sigma_{j} \subset \R^{d}$, a sequence of pairwise disjoint closed subsets $K_{j} \subset \Sigma_{j}$, and a sequence of positive integers $k_{j}$ such that $\sum_{j=1}^{\infty} k_{j} \Ha^{m}(K_{j}) < \infty$ and
\begin{equation} \label{ir_repr}
T = \sum_{j=1}^{\infty} k_{j} \llbracket K_{j} \rrbracket.
\end{equation}
Now, if $u \colon \Omega \to \A_{Q}(\R^{n})$ is Lipschitz and proper, we define the push-forward of $T$ through $u$ by setting 
\begin{equation} \label{Q_pf_ir:repr}
u_{\sharp} T := \sum_{j=1}^{\infty} k_{j} {\bf T}_{u_j},
\end{equation}
where $u_{j} := u|_{K_{j}}$. We record the properties of $u_{\sharp} T$ in the following proposition.

\begin{proposition}[$Q$-valued push-forward of rectifiable currents.]
The integer rectifiable current $u_{\sharp} T \in \Rc_{m}(\R^{n})$ defined in \eqref{Q_pf_ir:repr} is independent of the particular representation \eqref{ir_repr} of $T$. If $T = \llbracket B, \vec{\tau}, \theta \rrbracket$, then $u_{\sharp} T$ acts on forms $\omega \in \D^{m}(\R^{n})$ as follows:
\begin{equation} \label{Q_pf_ir:eq}
(u_{\sharp} T)(\omega) = \int_{B} \sum_{l=1}^{Q} \langle \omega(u_{l}(p)), Du_{l}(p)_{\sharp} \vec{\tau}(p) \rangle \, \theta(p) \, d\Ha^{m}(p).
\end{equation}
Moreover, $u_{\sharp} T$ can be represented by $u_{\sharp} T = \llbracket {\rm Im}(u|_{B}), \vec{\eta}, \Theta \rrbracket$, where
\begin{itemize}
\item[$(R1)'$] ${\rm Im}(u|_{B}) = \bigcup_{j=1}^{\infty} {\rm Im}(u_{j})$ is an $m$-rectifiable set in $\R^{n}$;

\item[$(R2)'$] $\vec{\eta}$ is a Borel unit $m$-vector field orienting ${\rm Im}(u|_{B})$; moreover, if $K_{j} = \bigcup_{i \in \N} K^{i}_{j}$ is a countable partition of $K_{j} \subset \Sigma_{j}$ in measurable subsets associated to a Lipschitz selection $u|_{K^{i}_{j}} = \sum_{l} \llbracket (u^{i}_{j})^{l} \rrbracket$ of $u$ as in Proposition \ref{Lip_sel}, then for $\Ha^{m}$-a.e. $y \in {\rm Im}(u|_{B})$ one has that
\begin{equation}
\frac{D(u^{i}_{j})^{l}(p)_{\sharp} \vec{\tau}(p)}{|D(u^{i}_{j})^{l}(p)_{\sharp} \vec{\tau}(p)|} = \pm \vec{\eta}(y)
\end{equation}
for all $j, i, l,p$ such that $(u^{i}_{j})^{l}(p) = y$;

\item[$(R3)'$] for $\Ha^{m}$-a.e. $y \in {\rm Im}(u|_{B})$, the (Borel) multiplicity function $\Theta$ equals
\begin{equation} 
\Theta(y) = \sum_{j,i,l,p \, \colon \, (u^{i}_{j})^{l}(p) = y} \theta(p) \left\langle \vec{\eta}(y), \frac{D(u^{i}_{j})^{l}(p)_{\sharp} \vec{\tau}(p)}{|D(u^{i}_{j})^{l}(p)_{\sharp} \vec{\tau}(p)|} \right\rangle.
\end{equation}
\end{itemize}
\end{proposition}

\begin{notazioni} \label{notazioni}
In the rest of this Appendix, we will use the symbol $u_{\sharp} T$ to denote the push-forward of a current $T \in \D_{m}(\Omega)$ through a multiple valued function $u \colon \Omega \to \A_{Q}(\R^{n})$ whenever such a push-forward is defined. The symbol ${\bf T}_{u}$ may be still used when it is understood that the push-forward operator is acting on the whole domain of $u$. In particular, if $\Sigma \subset \R^{d}$ is an $m$-dimensional Lipschitz submanifold and $u \colon \Sigma \to \A_{Q}(\R^{n})$ then the writings ${\bf T}_{u}$ and $u_{\sharp} \llbracket \Sigma \rrbracket$ are equivalent.
\end{notazioni}

We are now in the position to define $u_{\sharp}T$ when $T \in \F_{m}(\Omega)$. Let us fix the following hypotheses.

\begin{ipotesi} \label{hyp}
We will consider:
\begin{itemize}
\item a Lipschitz $Q$-valued function $u \colon \Omega \to \A_{Q}(\R^{n})$ defined in an open subset $\Omega \subset \R^{d}$;

\item a compact subset $K \subset \Omega$;

\item an integral flat $m$-chain $T \in \F_{m}(\R^{d})$ with $\spt(T) \subset {\rm int}K$.
\end{itemize}
\end{ipotesi}
 
Given $K$ and $T$ as in Assumptions \ref{hyp}, \cite[Theorem 4.2.22]{Federer69} there exists a sequence $\{ P_{j} \}_{j=1}^{\infty}$ of integral polyhedral $m$-chains supported in $K$ such that
\begin{equation} \label{pf_fc:1}
\Fl_{K}(T - P_{j}) \leq \frac{1}{j} \hspace{0.3cm} \mbox{ and } \hspace{0.3cm} \M(P_{j}) \leq \M(T) + \frac{1}{j}.
\end{equation} 

Now, integral polyhedral chains are a subclass of the class of integer rectifiable currents, as any $P_{j}$ can be written as the linear combination $P_{j} = \sum_{i=1}^{k_j} \beta_{ji} \llbracket \sigma_{ji} \rrbracket$ of a finite number of oriented simplexes $\llbracket \sigma_{ji} \rrbracket$ with coefficients $\beta_{ji} \in \Z$. Since we have a well defined notion of multi-valued push-forward of an integer rectifiable current, we can consider the currents
\begin{equation} \label{pf_fc:2}
u_{\sharp} P_{j} = \sum_{i=1}^{k_j} \beta_{ji} \, u_{\sharp} \llbracket \sigma_{ji} \rrbracket.
\end{equation}

We also know that the mass of $u_{\sharp}P_{j}$ can be estimated by
\begin{equation} \label{pf_fc:3}
\M(u_{\sharp}P_{j}) \leq C \M(P_{j}),
\end{equation}
where $C$ is a constant depending on $\Lip(u)$, and Theorem \ref{pf_bdry:thm} guarantees that
\begin{equation} \label{pf_fc:4}
\partial (u_{\sharp} P_{j}) = u_{\sharp} (\partial P_{j}).
\end{equation}

Clearly, $\{ P_{j} \}$ is a Cauchy sequence with respect to the flat distance $\Fl_{K}$. Indeed, for any $j,h \in \N$ one can explicitly estimate
\begin{equation} \label{pf_fc:4bis}
\Fl_{K}(P_{j} - P_{h}) \leq \Fl_{K}(P_{j} - T) + \Fl_{K}(T - P_{h}) \leq \frac{1}{j} + \frac{1}{h}.
\end{equation}
Now, we have the following
\begin{theorem}[Push-forward of a flat chain] \label{pf_fc:thm}
Let $u$, $K$ and $T$ be as in Assumptions \ref{hyp}. Then, for any open subset $W \Subset \Omega$ with $K \subset W$, for any compact $K' \subset \R^{n}$ containing ${\rm Im}(u|_{W}) = \bigcup_{p \in W} \spt(u(p))$, and for any sequence $\lbrace P_{j} \rbrace_{j=1}^{\infty}$ of integral polyhedral $m$-chains converging to $T$ with respect to $d_{\Fl_{K}}$, the sequence $\{ u_{\sharp}P_{j} \}_{j=1}^{\infty}$ is Cauchy with respect to $d_{\Fl_{K'}}$. Therefore, there exists an integral flat $m$-chain $Z \in \F_{m,K'}(\R^{n})$ such that $\Fl_{K'}(Z - u_{\sharp}P_{j}) \to 0$ as $j \uparrow \infty$. Such a $Z$ does not depend on the approximating sequence $P_{j}$ converging to $T$.
\end{theorem}

\begin{definition}
The current $Z \in \F_{m}(\R^{n})$ given by Theorem \ref{pf_fc:thm} is the push-forward of $T$ through $u$. Coherently with Notation \ref{notazioni}, we will set $Z = u_{\sharp}T$.
\end{definition}

The proof of Theorem \ref{pf_fc:thm} is a simple consequence of the following lemma, the proof of which can be found in \cite[Lemma 4.2.23]{Federer69}. 
\begin{lemma} \label{pf_fc:lem}
If $K \subset W \subset \R^{d}$ with $K$ compact, $W$ open, and $P \in \mathscr{P}_{m}(\R^{d})$ with $\spt(P) \subset K$, then the quantity
\begin{equation} \label{pf_fc:lem1}
G(P) := \inf\left\lbrace \M(P - \partial S) + \M(S) \, \colon \, S \in \mathscr{P}_{m+1}(\R^{d}) \mbox{ with } \spt(S) \subset W \right\rbrace
\end{equation}
does not exceed $\Fl_{K}(P)$.
\end{lemma}

\begin{proof}[Proof of Theorem \ref{pf_fc:thm}.]
Fix any open set $W \Subset \Omega$ with $K \subset W$, let $K' \subset \R^{n}$ be any compact set containing ${\rm Im}(u|_{W})$, and let $\lbrace P_{j} \rbrace_{j=1}^{\infty}$ be any sequence of integral polyhedral $m$-chains supported in $K$ and satisfying \eqref{pf_fc:1}. For any $j, h \in \N$, consider the current $P_{j} - P_{h} \in \mathscr{P}_{m}(\R^{d})$, and notice that $\spt(P_{j} - P_{h}) \subset K$. For any choice of polyhedral currents $R \in \mathscr{P}_{m}(\R^{d})$, $S \in \mathscr{P}_{m+1}(\R^{d})$ with $\spt(R) \cup \spt(S) \subset W$ such that
\begin{equation} \label{pf_fc:5}
P_{j} - P_{h} = R + \partial S,
\end{equation}
Theorem \ref{pf_bdry:thm} guarantees that
\begin{equation} \label{pf_fc:6}
u_{\sharp}P_{j} - u_{\sharp}P_{h} = u_{\sharp}R + \partial(u_{\sharp}S).
\end{equation}
Since $u_{\sharp}R$ and $u_{\sharp}S$ are rectifiable currents supported in $K'$, one has
\begin{equation} \label{pf_fc:7}
\begin{split}
\Fl_{K'}(u_{\sharp}P_{j} - u_{\sharp}P_{h}) &\leq \M(u_{\sharp}R) + \M(u_{\sharp}S) \\
&\leq C \left( \M(R) + \M(S) \right),
\end{split}
\end{equation}
for some constant $C$ depending on $\Lip(u)$. Taking the infimum among all integral polyhedral currents $R$ and $S$ supported in $W$ such that \eqref{pf_fc:5} holds, we immediately conclude from Lemma \ref{pf_fc:lem} that
\begin{equation} \label{pf_fc:8}
\Fl_{K'}(u_{\sharp}P_{j} - u_{\sharp}P_{h}) \leq C G(P_{j} - P_{h}) \leq C \Fl_{K}(P_{j} - P_{h}) \leq \frac{C}{j} + \frac{C}{h}.
\end{equation}
This proves that the sequence $\lbrace u_{\sharp} P_{j} \rbrace_{j=1}^{\infty}$ is Cauchy with respect to $d_{\Fl_{K'}}$ and, thus, has a limit $Z \in \F_{m,K'}(\R^{n})$. In order to see that the limit does not depend on the approximating sequence $\lbrace P_{j} \rbrace$, consider two sequences of integral polyhedral $m$-currents $\lbrace P_{j} \rbrace$ and $\lbrace \tilde{P}_{j} \rbrace$ both approximating $T$ in the $\Fl_{K}$ distance, and assume that $u_{\sharp} P_{j}$ and $u_{\sharp} \tilde{P}_{j}$ flat converge to $Z$ and $\tilde{Z}$ respectively. For any $\varepsilon > 0$, let $j_{0} \in \N$ be such that both $\Fl_{K}(T - P_{j_0}) + \Fl_{K}(T - \tilde{P}_{j_0}) < \varepsilon$ and $\Fl_{K'}(Z - u_{\sharp}P_{j_0}) + \Fl_{K'}(\tilde{Z} - u_{\sharp} \tilde{P}_{j_0}) < \varepsilon$. Then, we can estimate:
\begin{equation} \label{pf_fc:9}
\begin{split}
\Fl_{K'}(Z - \tilde{Z}) &\leq \Fl_{K'}(Z - u_{\sharp}P_{j_0}) + \Fl_{K'}(u_{\sharp}P_{j_0} - u_{\sharp}\tilde{P}_{j_0}) + \Fl_{K'}(u_{\sharp}\tilde{P}_{j_0} - \tilde{Z}) \\
&\leq  \varepsilon + \Fl_{K'}(u_{\sharp}P_{j_0} - u_{\sharp}\tilde{P}_{j_0})
\end{split}
\end{equation}
On the other hand, applying the same argument that we have used above to prove \eqref{pf_fc:8} to $P_{j_0} - \tilde{P}_{j_0} \in \mathscr{P}_{m}(\R^{d})$ shows that
\begin{equation} \label{pf_fc:10}
\Fl_{K'}(u_{\sharp} P_{j_0} - u_{\sharp}\tilde{P}_{j_0}) \leq C \Fl_{K}(P_{j_0} - \tilde{P}_{j_0}) \leq C \varepsilon.
\end{equation}
Combining \eqref{pf_fc:9} and \eqref{pf_fc:10}, and letting $\varepsilon \downarrow 0$ yields that $Z = \tilde{Z}$.
\end{proof}

\begin{corollary}
Let $u$, $K$ and $T$ be as in Assumption \ref{hyp}. If $Z = u_{\sharp}T$, then $\partial Z = u_{\sharp}(\partial T)$.
\end{corollary}

\begin{proof}
Let $W \Subset \Omega$ and $K' \subset \R^{n}$ be as in Theorem \ref{pf_fc:thm}, and let $\lbrace P_{j} \rbrace_{j=1}^{\infty}$ be any sequence of integral polyhedral $m$-chains $\Fl_{K}$-converging to $T$. Then, by Theorem \ref{pf_fc:thm} $Z$ is the $\Fl_{K'}$-limit of the currents $u_{\sharp}P_{j}$. Hence, since in general $\Fl_{K}(\partial T) \leq \Fl_{K}(T)$, we also have that $\partial Z$ is the $\Fl_{K'}$-limit of the currents $\partial (u_{\sharp}P_{j}) = u_{\sharp}(\partial P_{j})$ by Theorem \ref{pf_bdry:thm}. On the other hand, since the $\partial P_{j}$'s are a sequence of integral polyhedral $(m-1)$-chains which $\Fl_{K}$-approximates $\partial T$, the sequence $u_{\sharp}(\partial P_{j})$ necessarily $\Fl_{K'}$-converges to $u_{\sharp}(\partial T)$. The claim follows by uniqueness of the limit.
\end{proof}

%
%
%
%
%
%
%
%

\bibliographystyle{aomalpha}
\bibliography{References}

\providecommand{\bysame}{\leavevmode\hbox to3em{\hrulefill}\thinspace}
\providecommand{\noopsort}[1]{}
\providecommand{\mr}[1]{\href{http://www.ams.org/mathscinet-getitem?mr=#1}{MR~#1}}
\providecommand{\zbl}[1]{\href{http://www.zentralblatt-math.org/zmath/en/search/?q=an:#1}{Zbl~#1}}
\providecommand{\jfm}[1]{\href{http://www.emis.de/cgi-bin/JFM-item?#1}{JFM~#1}}
\providecommand{\arxiv}[1]{\href{http://www.arxiv.org/abs/#1}{arXiv~#1}}
\providecommand{\doi}[1]{\url{https://doi.org/#1}}
\providecommand{\MR}{\relax\ifhmode\unskip\space\fi MR }
\providecommand{\MRhref}[2]{%
  \href{http://www.ams.org/mathscinet-getitem?mr=#1}{#2}
}
\providecommand{\href}[2]{#2}
\begin{thebibliography}{DHMS19b}

\bibitem[{All}13]{Allard}
\bgroup\scshape{}W.~K. {Allard}\egroup{}, {Some useful techniques for dealing
  with multiple valued functions},  \emph{Unpublished note} (2013).

\bibitem[Alm00]{Almgren00}
\bgroup\scshape{}F.~J. Almgren, Jr.\egroup{}, \emph{Almgren's big regularity
  paper}, \emph{World Scientific Monograph Series in Mathematics} \textbf{1},
  World Scientific Publishing Co., Inc., River Edge, NJ, 2000, $Q$-valued
  functions minimizing Dirichlet's integral and the regularity of
  area-minimizing rectifiable currents up to codimension 2, With a preface by
  Jean E. Taylor and Vladimir Scheffer. \mr{1777737 (2003d:49001)}.

\bibitem[AGS08]{AGS08}
\bgroup\scshape{}L.~Ambrosio\egroup{}, \bgroup\scshape{}N.~Gigli\egroup{}, and
  \bgroup\scshape{}G.~Savar{{\'e}}\egroup{}, \emph{Gradient flows in metric
  spaces and in the space of probability measures}, second ed., \emph{Lectures
  in Mathematics ETH Z{\"u}rich}, Birkh{\"a}user Verlag, Basel, 2008.
  \mr{2401600}.

\bibitem[AK00]{AK00}
\bgroup\scshape{}L.~Ambrosio\egroup{} and
  \bgroup\scshape{}B.~Kirchheim\egroup{}, Currents in metric spaces,
  \emph{Acta Math.} \textbf{185} no.~1 (2000), 1--80. \mr{1794185}.
  \doi{10.1007/BF02392711}.

\bibitem[DL16]{DeLellis2015}
\bgroup\scshape{}C.~De~Lellis\egroup{}, The size of the singular set of
  area-minimizing currents,  in \emph{Surveys in differential geometry 2016.
  {A}dvances in geometry and mathematical physics}, \emph{Surv. Differ. Geom.}
  \textbf{21}, Int. Press, Somerville, MA, 2016, pp.~1--83. \mr{3525093}.

\bibitem[DDHM]{DLDPHM}
\bgroup\scshape{}C.~{De Lellis}\egroup{}, \bgroup\scshape{}G.~{De
  Philippis}\egroup{}, \bgroup\scshape{}J.~Hirsch\egroup{}, and
  \bgroup\scshape{}A.~Massaccesi\egroup{}, On the boundary behavior of
  mass-minimizing integral currents. Available at
  \url{https://arxiv.org/pdf/1809.09457.pdf}.

\bibitem[DHMS19a]{DLHMS_linear}
\bgroup\scshape{}C.~{De Lellis}\egroup{}, \bgroup\scshape{}J.~Hirsch\egroup{},
  \bgroup\scshape{}A.~Marchese\egroup{}, and
  \bgroup\scshape{}S.~Stuvard\egroup{}, Area minimizing currents mod $2q$:
  linear regularity theory,  \emph{In preparation} (2019).

\bibitem[DHMS19b]{DLHMS}
\bgroup\scshape{}C.~{De Lellis}\egroup{}, \bgroup\scshape{}J.~Hirsch\egroup{},
  \bgroup\scshape{}A.~Marchese\egroup{}, and
  \bgroup\scshape{}S.~Stuvard\egroup{}, {Regularity of area minimizing currents
  mod $p$},  \emph{In preparation} (2019).

\bibitem[DLS11]{DLS11a}
\bgroup\scshape{}C.~De~Lellis\egroup{} and
  \bgroup\scshape{}E.~Spadaro\egroup{}, {$Q$}-valued functions revisited,
  \emph{Mem. Amer. Math. Soc.} \textbf{211} no.~991 (2011), vi+79. \mr{2663735
  (2012k:49112)}.  \doi{10.1090/S0065-9266-10-00607-1}.

\bibitem[DLS14]{DLS14}
\bgroup\scshape{}C.~De~Lellis\egroup{} and
  \bgroup\scshape{}E.~Spadaro\egroup{}, Regularity of area minimizing currents
  {I}: gradient {$L^p$} estimates,  \emph{Geom. Funct. Anal.} \textbf{24} no.~6
  (2014), 1831--1884. \mr{3283929}.  \doi{10.1007/s00039-014-0306-3}.

\bibitem[DS15]{DLS13a}
\bgroup\scshape{}C.~{De Lellis}\egroup{} and
  \bgroup\scshape{}E.~{Spadaro}\egroup{}, {Multiple valued functions and
  integral currents},  \emph{Ann. Sc. Norm. Super. Pisa Cl. Sci. (5)}
  \textbf{XIV} (2015), 1239--1269.

\bibitem[DLS16a]{DLS13b}
\bgroup\scshape{}C.~De~Lellis\egroup{} and
  \bgroup\scshape{}E.~Spadaro\egroup{}, Regularity of area minimizing currents
  {II}: center manifold,  \emph{Ann. of Math. (2)} \textbf{183} no.~2 (2016),
  499--575. \mr{3450482}.  \doi{10.4007/annals.2016.183.2.2}.

\bibitem[DLS16b]{DLS13c}
\bgroup\scshape{}C.~De~Lellis\egroup{} and
  \bgroup\scshape{}E.~Spadaro\egroup{}, Regularity of area minimizing currents
  {III}: blow-up,  \emph{Ann. of Math. (2)} \textbf{183} no.~2 (2016),
  577--617. \mr{3450483}.  \doi{10.4007/annals.2016.183.2.3}.

\bibitem[DSS]{DLSS3}
\bgroup\scshape{}C.~{De Lellis}\egroup{}, \bgroup\scshape{}E.~Spadaro\egroup{},
  and \bgroup\scshape{}L.~Spolaor\egroup{}, Regularity theory for
  $2$-dimensional almost minimal currents iii: blowup. Available at
  \url{https://arxiv.org/pdf/1508.05510.pdf}.

\bibitem[DLSS17]{DLSS2}
\bgroup\scshape{}C.~De~Lellis\egroup{}, \bgroup\scshape{}E.~Spadaro\egroup{},
  and \bgroup\scshape{}L.~Spolaor\egroup{}, Regularity theory for 2-dimensional
  almost minimal currents {II}: {B}ranched center manifold,  \emph{Ann. PDE}
  \textbf{3} no.~2 (2017), Art. 18, 85. \mr{3712561}.
  \doi{10.1007/s40818-017-0035-7}.

\bibitem[DLSS18]{DLSS1}
\bgroup\scshape{}C.~De~Lellis\egroup{}, \bgroup\scshape{}E.~Spadaro\egroup{},
  and \bgroup\scshape{}L.~Spolaor\egroup{}, Regularity theory for
  {$2$}-dimensional almost minimal currents {I}: {L}ipschitz approximation,
  \emph{Trans. Amer. Math. Soc.} \textbf{370} no.~3 (2018), 1783--1801.
  \mr{3739191}.  \doi{10.1090/tran/6995}.

\bibitem[Fed69]{Federer69}
\bgroup\scshape{}H.~Federer\egroup{}, \emph{Geometric measure theory},
  \emph{Die Grundlehren der mathematischen Wissenschaften, Band 153},
  Springer-Verlag New York Inc., New York, 1969. \mr{0257325 (41 \#1976)}.

\bibitem[GMS98]{GMS98}
\bgroup\scshape{}M.~Giaquinta\egroup{}, \bgroup\scshape{}G.~Modica\egroup{},
  and \bgroup\scshape{}J.~r. Sou\v{c}ek\egroup{}, \emph{Cartesian currents in
  the calculus of variations. {I}}, \emph{Ergebnisse der Mathematik und ihrer
  Grenzgebiete. 3. Folge. A Series of Modern Surveys in Mathematics [Results in
  Mathematics and Related Areas. 3rd Series. A Series of Modern Surveys in
  Mathematics]} \textbf{37}, Springer-Verlag, Berlin, 1998, Cartesian currents.
  \mr{1645086}.  \doi{10.1007/978-3-662-06218-0}.

\bibitem[KP08]{KP08}
\bgroup\scshape{}S.~G. Krantz\egroup{} and \bgroup\scshape{}H.~R.
  Parks\egroup{}, \emph{Geometric integration theory}, \emph{Cornerstones},
  Birkh{\"a}user Boston, Inc., Boston, MA, 2008. \mr{2427002}.
  \doi{10.1007/978-0-8176-4679-0}.

\bibitem[Sim83]{Sim83}
\bgroup\scshape{}L.~Simon\egroup{}, \emph{Lectures on geometric measure
  theory}, \emph{Proceedings of the Centre for Mathematical Analysis,
  Australian National University} \textbf{3}, Australian National University,
  Centre for Mathematical Analysis, Canberra, 1983. \mr{756417}.

\bibitem[Stu19]{SS17b}
\bgroup\scshape{}S.~Stuvard\egroup{}, Multiple valued {J}acobi fields,
  \emph{Calc. Var. Partial Differential Equations} \textbf{58} no.~3 (2019),
  58:92. \mr{3948282}.  \doi{10.1007/s00526-019-1545-9}.

\bibitem[Vil03]{Vil03}
\bgroup\scshape{}C.~Villani\egroup{}, \emph{Topics in optimal transportation},
  \emph{Graduate Studies in Mathematics} \textbf{58}, American Mathematical
  Society, Providence, RI, 2003. \mr{1964483}.  \doi{10.1007/b12016}.

\end{thebibliography}

\Addresses

\end{document}